\documentclass[11pt,reqno]{amsart}
\usepackage{amsmath,amssymb,amsthm,amsfonts,mathtools}
\usepackage[margin=1.2in]{geometry}
\usepackage{microtype}
\usepackage{needspace}
\usepackage{comment}
\usepackage[%backref=page,
colorlinks=true,urlcolor=blue,linkcolor=blue,citecolor=magenta]{hyperref}
\usepackage{amsrefs}
\hypersetup{
 pdftitle={Alexandrov-Fenchel inequalities for convex domains in the sphere},
 pdfsubject={Spherical quermassintegrals and a globally constrained curvature flow},
 pdfkeywords={Alexandrov-Fenchel, sphere, quermassintegral, curvature flow, polar duality}
}
\allowdisplaybreaks
\numberwithin{equation}{section}

\newtheorem{theorem}{Theorem}[section]
\newtheorem{proposition}[theorem]{Proposition}
\newtheorem{lemma}[theorem]{Lemma}

\theoremstyle{remark}

\begin{document}

\title[Alexandrov--Fenchel inequalities in the sphere]{Alexandrov--Fenchel inequalities for convex domains in the sphere}

\author[T. Luo]{Tianci Luo}
\author[Y. Wei]{Yong Wei}
\author[R. Zhou]{Rong Zhou}
\address{School of Mathematical Sciences, University of Science and Technology of China, Hefei 230026, P.R. China}
\email{Luo\_tianci@mail.ustc.edu.cn}
\email{yongwei@ustc.edu.cn}
\email{zhourong@mail.ustc.edu.cn}

\date{\today}

\begin{abstract}
We prove the Alexandrov--Fenchel inequalities between any two
spherical quermassintegrals for smooth weakly convex domains contained in an open hemisphere, with equality if and only if the domain is a geodesic ball. For strictly convex hypersurfaces, we introduce a globally constrained curvature flow which preserves one quermassintegral and decreases the next one. We establish uniform curvature estimates by combining a pinching estimate, a support function argument, and spherical polarity. As a consequence, the flow exists for all time and converges smoothly and exponentially to a geodesic sphere. The monotonicity of the quermassintegrals gives the full family of Alexandrov--Fenchel inequalities. A short-time mean curvature flow approximation and a localized rigidity argument extend the result, including the equality characterization, to weakly convex domains.
\end{abstract}

\maketitle
\tableofcontents

\section{Introduction}\label{sec:intro}
The Alexandrov--Fenchel inequalities are fundamental inequalities in
convex geometry. In Euclidean space, they give sharp comparisons
between the quermassintegrals of a convex body, with equality
precisely for balls. We refer to Schneider \cite{Schneider} for the
classical theory. Curvature flows provide an important approach to
these inequalities by deforming a hypersurface towards a sphere while
preserving one geometric quantity and decreasing another. Important
examples include the volume-preserving mean curvature flow of Huisken
\cite{Huisken}, the mixed-volume-preserving curvature flows of McCoy
\cite{McCoy}, and the inverse curvature flow of Guan and Li
\cite{GuanLi2009}. In particular, Guan and Li proved the
quermassintegral inequalities for suitable $k$-convex star-shaped
domains in Euclidean space.

For convex domains in space forms, the quermassintegrals are naturally
defined by integral geometry. Their curvature integral expressions
contain additional lower order terms arising from the ambient
curvature. We refer to Santal\'o \cite{Santalo} and Solanes
\cites{SolanesThesis,Solanes2006} for the corresponding integral
geometric formulas.

Let $\mathbb N^{n+1}(K)$ denote the simply connected space form of constant sectional curvature $K\in\{-1,0,1\}$. Let $\Omega\subset\mathbb N^{n+1}(K)$ be a smooth bounded domain with boundary $M=\partial\Omega$, and let $\sigma_k(\kappa)$ denote the $k$th elementary symmetric function of the principal curvatures $\kappa =(\kappa_1,\ldots,\kappa_n)$ of $M$ with respect to the outward unit normal $\nu$. We use the normalization of the quermassintegrals adopted in \cites{BGL,CGLS}:
\begin{equation}\label{eq:Adefinition}
\begin{aligned}
 \mathcal A_{-1}(\Omega)
 &=\operatorname{Vol}(\Omega),
 \qquad
 \mathcal A_0(\Omega)=|M|,\\
 \mathcal A_1(\Omega)
 &=\int_M\sigma_1(\kappa)\,d\mu
   +nK\mathcal A_{-1}(\Omega),\\
 \mathcal A_k(\Omega)
 &=\int_M\sigma_k(\kappa)\,d\mu
   +K\frac{n-k+1}{k-1}\mathcal A_{k-2}(\Omega),
 \qquad 2\leq k\leq n.
\end{aligned}
\end{equation}
Thus $K=0$, $K=-1$, and $K=1$ give the Euclidean, hyperbolic, and
spherical quermassintegrals, respectively.

For a geodesic ball $B_r\subset\mathbb N^{n+1}(K)$, set
\begin{equation}\label{eq:ballfunctions}
 f_{j}(r)=\mathcal A_j(B_r).
\end{equation}
The Alexandrov--Fenchel problem is to establish the sharp comparison
\begin{equation}\label{eq:introAF}
 \mathcal A_k(\Omega)
 \geq
 f_{k}\circ f_{\ell}^{-1}
 \bigl(\mathcal A_\ell(\Omega)\bigr),
 \qquad
 -1\leq\ell<k\leq n-1,
\end{equation}
with equality precisely for geodesic balls.

We first recall the hyperbolic case. A smooth bounded domain in $\mathbb H^{n+1}=\mathbb N^{n+1}(-1)$ is called horospherically convex if its principal curvatures satisfy $\kappa_i\geq1$. Wang and Xia \cite{WangXia} proved the full family of Alexandrov--Fenchel inequalities \eqref{eq:introAF} under this assumption. Hu, Li, and the second author \cite{HuLiWei} later gave another proof by locally constrained curvature flows. Related hyperbolic Alexandrov--Fenchel inequalities were obtained by Ge, Wang, and Wu \cite{GeWangWu}.

Several important cases have been established under weaker convexity assumptions. Brendle, Guan, and Li \cite{BGL} proved \eqref{eq:introAF} for convex domains with $-1\le \ell<k=n-1$. Andrews, Chen, and the second author \cite{AndrewsChenWei} proved \eqref{eq:introAF} for $\ell=-1$ and $0\le k\le n-1$ when the domain is convex and its boundary has positive intrinsic sectional curvature. More recently, Hu and Li \cite{HuLi2023} used the duality between hyperbolic and de Sitter space to prove \eqref{eq:introAF} for smooth strictly convex domains with $-1\le \ell<k=n-2$. Their duality formulas also yield Blaschke--Santal\'o-type inequalities in space forms.

The spherical problem is less complete. Guan and Li \cite{GuanLi} introduced a locally constrained mean curvature type flow in space forms which preserves the enclosed volume and decreases the surface area. Brendle, Guan, and Li \cite{BGL} introduced an inverse curvature type flow adapted to higher quermassintegrals. For smooth strictly
convex domains in $\mathbb S^{n+1}$, they proved \eqref{eq:introAF} for $-1\leq\ell<k=n-1$ by establishing long-time existence and
exponential convergence of the flow associated with the harmonic mean curvature quotient. For the remaining curvature quotients, a uniform positive lower bound for the quotient is not available from their spherical estimates, see \cite{GuanLiSurvey}*{Section~4.2}.

Chen, Guan, Li, and Scheuer \cite{CGLS} studied another fully nonlinear locally constrained flow in the sphere. They proved preservation of convexity and two-sided bounds for the curvature quotient. Their estimates, however, do not provide a uniform bound for the full second fundamental form. See also \cite{GuanLiSurvey}*{Section~4.2}.

Other special cases of the spherical Alexandrov--Fenchel inequalities have been obtained by different methods. Chen and Sun \cite{ChenSun} proved \eqref{eq:introAF} for $\ell=k-2$, $1\leq k\leq n-1$, for embedded, closed, connected convex $C^2$-hypersurfaces contained in an open hemisphere. Makowski and Scheuer \cite{MakowskiScheuer} and the second author with Xiong \cite{WeiXiong} obtained related sharp inequalities for convex hypersurfaces by curvature flow methods. More recently, Chen \cite{Chen2025} established further inequalities between curvature integrals and quermassintegrals for embedded, closed, connected convex $C^2$-hypersurfaces.

There has also been recent progress based on quermassintegral-preserving flows. Cabezas-Rivas and Scheuer \cite{CabezasScheuer} constructed such a mean curvature flow for arbitrary smooth strictly convex hypersurfaces in the sphere and proved convergence to a geodesic sphere. The convergence of this flow does not by itself give the monotonicity of the next quermassintegral needed for the adjacent Alexandrov--Fenchel inequalities. Pan and Scheuer \cite{PanScheuer} introduced spherical horo-convexity and proved the full family \eqref{eq:introAF} under this stronger convexity assumption. Albert-Nicl\`os, Cabezas-Rivas, and Pan \cite{AlbertCabezasPan} subsequently proved long-time existence and exponential convergence for fully nonlinear quermassintegral-preserving flows in the same class.

The results above leave open the full family of Alexandrov--Fenchel inequalities for general convex domains in the sphere without a horo-convexity assumption. The purpose of this paper is to prove the full family for smooth weakly convex domains contained in an open hemisphere.

\subsection{Main results}\label{subsec:main}
From now on, the ambient space is $\mathbb S^{n+1}=\mathbb N^{n+1}(1)$. A smooth weakly convex domain $\Omega\subset\mathbb S^{n+1}$ means a compact geodesically convex body with nonempty interior, contained in an open hemisphere, whose boundary is smooth and has nonnegative principal curvatures $\kappa =(\kappa_1,\ldots,\kappa_n)$ with respect to the outward unit normal $\nu$. It is called strictly convex if all its principal curvatures are positive.

Our main result is the following Alexandrov--Fenchel inequality.

\begin{theorem}\label{thm:AF}
Let $\Omega\subset\mathbb S^{n+1}$ be a smooth weakly convex domain, where $n\geq2$. For every pair of integers $-1\leq \ell<k\leq n-1$,
\begin{equation}\label{eq:AF}
 \mathcal{A}_k(\Omega)\geq f_k \circ f_\ell^{-1}(\mathcal{A}_\ell(\Omega)).
\end{equation}
Equality holds if and only if $\Omega$ is a geodesic ball.
\end{theorem}

Notice that only weak convexity is assumed in Theorem~\ref{thm:AF}. The restriction $k\leq n-1$ is natural since $\mathcal A_n(\Omega)=|\mathbb S^n|$.

For strictly convex domains, Theorem~\ref{thm:AF} follows from a new globally constrained curvature flow. Let
\begin{equation*}
 E_k(\kappa)=\binom nk^{-1}\sigma_k(\kappa),
 \qquad
 E_0=1.
\end{equation*}
Fix $k\in\{0,\ldots,n-1\}$. Starting from a smooth strictly convex
embedding $X_0:M^n\to\mathbb S^{n+1}$, we consider
\begin{equation}\label{eq:mainflow}
\begin{gathered}
 \partial_tX=\left(\frac{\phi(t)}F-F\right)\nu,\qquad X(\cdot,0)=X_0,\\
 F=\frac{E_{k+1}(\kappa)}{E_k(\kappa)},\qquad
 \phi(t)=
 \frac{\displaystyle\int_{M_t}E_k(\kappa)F\,d\mu_t}
      {\displaystyle\int_{M_t}E_k(\kappa)/F\,d\mu_t}.
\end{gathered}
\end{equation}
By the choice of $\phi(t)$, the flow \eqref{eq:mainflow} preserves $\mathcal A_{k-1}$ and decreases $\mathcal A_k$. Standard parabolic theory gives a unique smooth solution for a short time for any smooth strictly convex initial hypersurface, see \cite{HuiskenPolden}. We prove the following long-time existence and convergence theorem.

\begin{theorem}\label{thm:flow}
Let $M^n$ be a closed manifold, with $n\geq2$, and let $X_0:M^n\to\mathbb S^{n+1}$ be a smooth embedding such that $M_0=X_0(M)$ bounds a smooth strictly convex domain $\Omega_0$. For each $k\in\{0,\ldots,n-1\}$, the flow \eqref{eq:mainflow} with initial value $X_0$ has a unique smooth solution for all $t\in[0,\infty)$. Each $M_t=X(M,t)$ is strictly convex, and $M_t$ converges smoothly and exponentially as $t\to\infty$ to a geodesic sphere of radius $R_\infty\in(0,\pi/2)$ determined by
\begin{equation*}
 f_{k-1}(R_\infty)=\mathcal A_{k-1}(\Omega_0).
\end{equation*}
\end{theorem}

\subsection{Discussion of the proof}

The main difficulty in proving Theorem~\ref{thm:flow} is to
obtain uniform two-sided curvature estimates under strict
convexity alone. Our choice of speed in \eqref{eq:mainflow}
is adapted both to the required monotonicity of the
quermassintegrals and to spherical polarity. The key observation
is that, after a change of time, the polar evolution has the
same speed form. This allows us to obtain the lower curvature
bound by applying an upper curvature estimate to the polar
hypersurface. For this argument to work, the upper estimate
must not require a priori bounds for the nonlocal factor
$\phi(t)$.

We establish this estimate using only the positivity of
$\phi(t)$. The concavity and inverse-concavity of
$G_a=F-a/F$ for every $a>0$ allow us to apply the algebraic
estimate of Andrews \cite{Andrews}, together with the lower
test function argument of Brendle, Choi, and Daskalopoulos
\cite{BCD}, to preserve strict convexity and obtain a uniform
curvature-ratio bound. This bound and the preserved
quermassintegral give a positive lower bound for the inradius
by the radius comparison of Cabezas-Rivas and Scheuer
\cite{CabezasScheuer}. We then apply a Tso-type argument to
$F/(u-d)$, where $u$ is the support function relative to an
interior point and $d>0$ is chosen from the inradius bound.
The terms involving $\phi(t)$ have a favorable sign, so the
resulting upper curvature bound is independent of its size.
Since the speed contains neither a support function nor a
radial weight, the interior point used in this estimate can
be chosen separately on successive time intervals without
changing the evolution equation.

To obtain the complementary lower bound, we use spherical
polarity \cites{Gerhardt,HuLi2023}. The polar hypersurface
satisfies
\begin{equation*}
 \mathcal W^*=\mathcal W^{-1},
 \qquad
 F^*=\frac{E_{n-k}(\mathcal W^*)}
           {E_{n-k-1}(\mathcal W^*)}
     =\frac1F.
\end{equation*}
After removing the tangential velocity and changing time,
its evolution becomes
\begin{equation*}
 \partial_\tau Y
 =\left(\frac{\phi^*}{F^*}-F^*\right)N,
 \qquad
 \frac{d\tau}{dt}=\phi(t),
 \qquad
 \phi^*(\tau(t))=\frac1{\phi(t)}.
\end{equation*}
The transformed normalization preserves
$\mathcal A_{n-k-1}$. Although the relation between the
quotient index and the preserved index differs from that
of the original flow, our upper estimate is formulated with
these indices independent. It therefore applies to the
polar evolution and bounds $\mathcal W^*$ from above.
The identity $\mathcal W^*=\mathcal W^{-1}$ then gives a
uniform positive lower bound for the original principal
curvatures. This obtains both curvature bounds from the same
estimate, without first controlling $\phi$ or $\phi^{-1}$.
Uniform bounds for $\phi$, and hence comparability of the
two time parameters, follow afterwards.

Once these curvature bounds are established, standard
parabolic regularity and continuation arguments give
higher-order estimates and long-time existence.
The positive ambient curvature also improves the preserved
pinching exponentially. Together with the derivative
estimates and the Codazzi equation, this gives exponential
decay of the normal speed and smooth convergence to a
single geodesic sphere. The monotonicity of the
quermassintegrals then yields the adjacent inequalities,
from which the full family follows by iteration.

For weakly convex domains, the inequalities follow by
approximation, but their equality characterization requires
a separate argument. If equality holds, variations supported
in the strictly convex region show that the relevant
curvature quotient is a positive constant there.
At a boundary point of this region where a principal
curvature vanishes, the lower test function comparison
and inverse-concavity control the derivative terms in the
Simons identity, while the positive ambient-curvature term
gives a contradiction. Thus the strictly convex region is
also closed, and the equality case reduces to that for
strictly convex domains.

The distinction from the locally constrained approach is
also illustrated in Appendix~\ref{sec:counterexample}.
For the flow of Chen, Guan, Li, and Scheuer \cite{CGLS},
we construct smooth strictly convex initial hypersurfaces
whose largest principal curvature becomes unbounded in
finite time when $1\leq k\leq n-1$. In the polar description,
this occurs when the least principal curvature reaches zero
while the polar solution remains smooth. Thus the missing
upper curvature estimate for that flow cannot hold under
strict convexity alone.

\subsection{Organization of the paper}
In Section~\ref{sec:prelim}, we collect some preliminaries on curvature functions, hypersurfaces and quermassintegrals in the sphere, and spherical polarity. We also derive the monotonicity and evolution formulas for the flow \eqref{eq:mainflow}. In Section~\ref{sec:flow}, we prove preservation of strict convexity and establish the curvature pinching estimates. In Section~\ref{sec:curvature}, we derive the uniform two-sided curvature estimates and prove the long-time existence of the flow. In Section~\ref{sec:convergence}, we prove the smooth exponential
convergence to a geodesic sphere and complete the proofs of Theorems~\ref{thm:flow} and~\ref{thm:AF}. In Appendix~\ref{sec:counterexample}, we give a finite-time curvature
blow-up example for the locally constrained flow of Chen, Guan, Li, and Scheuer \cite{CGLS}.

\subsection{On the use of AI}
\label{subsec:AI}

We initially approached the spherical Alexandrov--Fenchel
inequalities through locally constrained curvature flows.
In particular, we sought uniform $C^2$ estimates for the
flow of Chen, Guan, Li, and Scheuer \cite{CGLS}.
During this investigation, GPT-5.6 Pro (OpenAI) proposed
the curvature blow-up example presented in
Appendix~\ref{sec:counterexample}, which we subsequently
verified independently. This example shows that, when
$1\leq k\leq n-1$, smooth strictly convex initial hypersurfaces
can develop unbounded principal curvature in finite time
under that flow. It therefore rules out the upper curvature
estimate we had sought under strict convexity alone.
This led us to shift our focus from this locally constrained
approach to the globally constrained quermassintegral-preserving
flow \eqref{eq:mainflow} studied in this paper.

We also used ChatGPT (OpenAI) for language editing and for
comments on the clarity, exposition, and organization of
the manuscript. We independently assessed all AI-generated
suggestions and verified all mathematical arguments,
statements, calculations, and references. We take full
responsibility for the content of the manuscript.

\section{Preliminaries}\label{sec:prelim}
In this section, we collect some preliminaries on curvature functions, the geometry of hypersurfaces in the sphere, spherical quermassintegrals, and polarity. We also record the monotonicity and evolution formulas that will be used in the subsequent sections.

\subsection{Curvature functions}\label{subsec:symmetric}
We first recall the general curvature-function notation and the
inverse-concavity estimates used below. Let
\begin{equation*}
 \Gamma_+
 =\{\kappa=(\kappa_1,\ldots,\kappa_n)\in\mathbb R^n:
       \kappa_i>0\text{ for every }i\}.
\end{equation*}
A smooth symmetric function $\Phi$ on $\Gamma_+$ will also be regarded
as a function of positive definite symmetric matrices by evaluation at
their eigenvalues. For symmetric matrices $A$ and $\xi$, we write
\begin{equation*}
 \left.\frac{d}{ds}\Phi(A+s\xi)\right|_{s=0}
 =\dot\Phi^{ij}\xi_{ij},
 \qquad
 \left.\frac{d^2}{ds^2}\Phi(A+s\xi)\right|_{s=0}
 =\ddot\Phi^{ij,k\ell}\xi_{ij}\xi_{k\ell}.
\end{equation*}
If $A$ is diagonal with eigenvalues
$\kappa_1,\ldots,\kappa_n$, then
\begin{equation*}
 \dot\Phi^{ii}
 =\dot\Phi^i
 =\frac{\partial\Phi}{\partial\kappa_i},
 \qquad
 \dot\Phi^{ij}=0\quad\text{for }i\neq j.
\end{equation*}
We use the same convention for all curvature functions appearing below. Repeated tensor indices are summed from $1$ to $n$ unless otherwise stated.

Following \cite{Andrews}*{Section~2}, a smooth symmetric function $\Phi$ on $\Gamma_+$ is called \emph{inverse-concave} if
\begin{equation}\label{eq:inverseConcavityDefinition}
 A\longmapsto-\Phi(A^{-1})
\end{equation}
is concave on the cone of positive definite symmetric matrices. If $\Phi$ is inverse-concave, then \cite{Andrews}*{Lemma~3.4} gives, in a principal frame,
\begin{equation}\label{eq:inverseConcavity}
 \ddot\Phi^{ij,k\ell}\xi_{ij}\xi_{k\ell}
 +2\sum_{i,j}\frac{\dot\Phi^i}{\kappa_j}\xi_{ij}^2
 \geq0
\end{equation}
for every symmetric matrix $\xi$.

The next consequence of the algebraic estimate of Andrews \cite{Andrews}*{Theorem~4.1} will be used in the curvature pinching argument.

\begin{lemma}\label{lem:AndrewsGradient}
Let $\Phi$ be a smooth symmetric strictly increasing, concave, and
inverse-concave function on $\Gamma_+$. Let
\begin{equation*}
 A=\operatorname{diag}(\kappa_1,\ldots,\kappa_n)>0,
 \qquad
 H=\operatorname{tr}A=\sum_i\kappa_i.
\end{equation*}
Suppose that, for some $0<\varepsilon<1/n$,
\begin{equation*}
 \kappa_1=\cdots=\kappa_\mu=\varepsilon H,
 \qquad
 \kappa_j>\kappa_1\quad(j>\mu).
\end{equation*}
Let $T_{ijk}$ be a totally symmetric three-tensor satisfying
\begin{equation}\label{eq:algebraicNullData}
 T_{pij}
 =\varepsilon\left(\sum_qT_{pqq}\right)\delta_{ij}
 \qquad
 (1\leq p\leq n,\ 1\leq i,j\leq\mu).
\end{equation}
Then
\begin{equation}\label{eq:pinchingGradient}
 \ddot\Phi^{ij,k\ell}
 \left(
 T_{1ij}T_{1k\ell}
 -\varepsilon\sum_pT_{pij}T_{pk\ell}
 \right)
 +2\sum_i\sum_{j>\mu}
 \frac{\dot\Phi^i}{\kappa_j-\kappa_1}T_{i1j}^2
 \geq 0.
\end{equation}
\end{lemma}

\begin{proof}
We apply \cite{Andrews}*{Theorem~4.1} and compute explicitly the
supremum appearing there. Set $P=A-\varepsilon H\operatorname{Id}$ and
$Z_{ip}=T_{ip1}-\varepsilon\delta_{p1}\sum_qT_{iqq}$.
The condition \eqref{eq:algebraicNullData} gives the null-gradient
hypothesis of \cite{Andrews}*{Theorem~4.1} with $v=e_1$.
Thus
\begin{equation}\label{eq:AndrewsSup}
 \ddot\Phi^{ij,k\ell}
 \left(T_{1ij}T_{1k\ell}
       -\varepsilon\sum_pT_{pij}T_{pk\ell}\right)+2\sup_\Gamma\dot\Phi^{ij}
 \left(2\Gamma_i^pZ_{jp}-\Gamma_i^p\Gamma_j^qP_{pq}\right) \geq 0,
\end{equation}
where the supremum is over all real $n\times n$ matrices $\Gamma$.
In this frame,
\begin{equation*}
 P_{pq}=(\kappa_p-\kappa_1)\delta_{pq},
 \qquad
 \dot\Phi^{ij}=\dot\Phi^i\delta_{ij}.
\end{equation*}
Moreover, \eqref{eq:algebraicNullData} and the symmetry of $T$ imply
that $Z_{ip}=0$ for $p\leq\mu$, while $Z_{ip}=T_{ip1}$ for
$p>\mu$. Hence
\begin{align*}
&\dot\Phi^{ij}
 \left(2\Gamma_i^pZ_{jp}
       -\Gamma_i^p\Gamma_j^qP_{pq}\right)\\
=&\sum_i\sum_{p>\mu}\dot\Phi^i
 \left(
 2\Gamma_i^pT_{ip1}
 -(\kappa_p-\kappa_1)(\Gamma_i^p)^2
 \right)\\
=&\sum_i\sum_{p>\mu}\dot\Phi^i
 \left[
 \frac{T_{ip1}^2}{\kappa_p-\kappa_1}
 -(\kappa_p-\kappa_1)
 \left(
 \Gamma_i^p-\frac{T_{ip1}}{\kappa_p-\kappa_1}
 \right)^2
 \right].
\end{align*}
Therefore the supremum is attained by taking
\begin{equation*}
 \Gamma_i^p=\frac{T_{ip1}}{\kappa_p-\kappa_1}
 \quad\text{for }p>\mu,
 \qquad
 \Gamma_i^p=0
 \quad\text{for }p\leq\mu,
\end{equation*}
and hence
\begin{equation*}
 \sup_\Gamma\dot\Phi^{ij}
 \left(2\Gamma_i^pZ_{jp}
       -\Gamma_i^p\Gamma_j^qP_{pq}\right)
 =
 \sum_i\sum_{p>\mu}
 \frac{\dot\Phi^i}{\kappa_p-\kappa_1}T_{ip1}^2.
\end{equation*}
Substituting this into \eqref{eq:AndrewsSup} and using
$T_{ip1}=T_{i1p}$ proves \eqref{eq:pinchingGradient}.
\end{proof}

We now specialize to the elementary symmetric functions. For
$\kappa=(\kappa_1,\ldots,\kappa_n)\in\mathbb R^n$, let
\begin{equation*}
 E_k(\kappa)=\binom nk^{-1}\sigma_k(\kappa)
 =\binom nk^{-1}\sum_{i_1<\cdots<i_k}
                    \kappa_{i_1}\cdots\kappa_{i_k},
 \qquad 1\leq k\leq n.
\end{equation*}
We set $E_0=1$ and $E_{n+1}=0$. The standard identities in
\cite{HuLiWei}*{Lemma~2.1} give
\begin{equation}\label{eq:Ejidentities}
\begin{aligned}
 \dot E_k^{ij}A_{ij}&=kE_k,\\
 \dot E_k^{ij}\delta_{ij}&=kE_{k-1},\\
 \dot E_k^{ij}(A^2)_{ij}&=nE_1E_k-(n-k)E_{k+1},
\end{aligned}
\qquad 1\leq k\leq n.
\end{equation}
For $1\leq j\leq k\leq n-1$, the Newton--MacLaurin inequality
\cite{HuLiWei}*{Lemma~2.2} is
\begin{equation}\label{eq:NewtonMaclaurin}
 E_{k+1}E_{j-1}\leq E_kE_j,
\end{equation}
with equality if and only if all eigenvalues are equal.

For $0\leq k\leq n-1$, set
\begin{equation}\label{eq:Fdefinition}
 F=\frac{E_{k+1}}{E_k}.
\end{equation}
By \cite{Andrews}*{Theorem~2.6}, $F$ is smooth, symmetric,
homogeneous of degree one, strictly increasing, concave, and
inverse-concave on $\Gamma_+$. We shall also use its reciprocal dual
\begin{equation}\label{eq:dualF}
 F_*(A)=\frac1{F(A^{-1})}
       =\frac{E_{n-k}(A)}{E_{n-k-1}(A)}.
\end{equation}
Thus $F_*$ is again an adjacent quotient and is concave on $\Gamma_+$
by the same theorem. Homogeneity and monotonicity give
\begin{equation}\label{eq:minmaxF}
 \dot F^{ij}A_{ij}=F,
 \qquad
 F(\operatorname{Id})=1,
 \qquad
 \kappa_{\min}\leq F\leq\kappa_{\max}.
\end{equation}
Moreover, \cite{HuLiWei}*{Corollary~2.3} gives
\begin{equation}\label{eq:quotientTraces}
 1\leq\dot F^{ij}\delta_{ij}\leq k+1,
 \qquad
 F^2\leq\dot F^{ij}(A^2)_{ij}\leq(n-k)F^2.
\end{equation}

For $a>0$, define
\begin{equation}\label{eq:Ga}
 G_a=F-\frac{a}{F}.
\end{equation}

\begin{lemma}\label{lem:Gconcavity}
For every $a>0$, the function $G_a$ is smooth, symmetric, strictly
increasing, concave, and inverse-concave on $\Gamma_+$.
\end{lemma}

\begin{proof}
The functions $s\mapsto s-a/s$ and $s\mapsto as-1/s$ are increasing and concave on $(0,\infty)$. Since $F$ and $F_*$ are concave, so are $G_a$ and
\begin{equation*}
 -G_a(A^{-1})=aF_*(A)-F_*(A)^{-1}.
\end{equation*}
Moreover, $\dot G_a^i=(1+a/F^2)\dot F^i>0$.
\end{proof}

\subsection{Hypersurfaces in the sphere}\label{subsec:geometry}
Fix a point $p\in\mathbb S^{n+1}$ and let $r$ denote the geodesic distance from
$p$. Away from $p$ and its antipodal point, the unit round sphere is
represented in geodesic polar coordinates as the warped product
$(0,\pi)\times\mathbb S^n$ equipped with the metric
\begin{equation*}
 \overline g=dr^2+\sin^2r\,\sigma,
 \qquad 0<r<\pi,
\end{equation*}
where $\sigma$ is the standard round metric on $\mathbb S^n$. We denote by
$\overline\nabla$ the Levi-Civita connection of $\overline g$. The
radial vector field $\sin r\,\partial_r$ satisfies
\begin{equation}
 \overline\nabla_U(\sin r\,\partial_r)=\cos r\,U   \label{equ-conformal}
\end{equation}
for every tangent vector $U$ on $\mathbb S^{n+1}$.

Let $X:M^n\to\mathbb S^{n+1}$ be a smooth closed hypersurface with outward unit
normal $\nu$. We denote by $g$, $\nabla$, and $d\mu$ the induced
metric, its Levi-Civita connection, and the area element, respectively.
Let $\{x^1,\ldots,x^n\}$ be local coordinates on $M$ and write
\begin{equation*}
 \partial_i=X_*\left(\frac{\partial}{\partial x^i}\right).
\end{equation*}
The induced metric and the second fundamental form are defined by
\begin{equation*}
 g_{ij}=\overline g(\partial_i,\partial_j),
 \qquad
 h_{ij}=\overline g(\overline\nabla_{\partial_i}\nu,\partial_j).
\end{equation*}
The associated Weingarten map is the $g$-self-adjoint endomorphism
\begin{equation*}
 \mathcal W=(h^i{}_j),\qquad h^i{}_j=g^{ik}h_{kj}.
\end{equation*}
Its eigenvalues are the principal curvatures, which we order as
$\kappa_1\leq\cdots\leq\kappa_n$. Although the matrix $(h^i{}_j)$
need not be symmetric in arbitrary coordinates, $\mathcal W$ is
self-adjoint with respect to $g$. Hence it is diagonalizable with real
eigenvalues. We also use $A$ for the second fundamental form and write
\begin{equation*}
    H=\operatorname{tr}\mathcal W,\qquad \mathring{\mathcal W}=\mathcal W-\frac{H}{n}\operatorname{Id},
    \qquad
    |A|^2=h^i{}_jh^j{}_i,\qquad
    (h^2)_{ij}=h_{ik}g^{k\ell}h_{\ell j}.
\end{equation*}
Thus a geodesic sphere of radius $r<\pi/2$, with the outward normal of
the enclosed ball, has $\mathcal W=\cot r\,\operatorname{Id}$. The Gauss and Codazzi
equations are
\begin{equation*}
 R_{ijk\ell}=g_{ik}g_{j\ell}-g_{i\ell}g_{jk}+h_{ik}h_{j\ell}-h_{i\ell}h_{jk},
 \qquad \nabla_i h_{jk}=\nabla_j h_{ik}.
\end{equation*}
For a curvature function, its matrix derivatives on $M$ are evaluated
in an orthonormal frame. Equivalently, since $\mathcal W=g^{-1}h$ is
$g$-self-adjoint, one may regard the curvature function as an
$O(n)$-invariant function of the symmetric matrix
$g^{-1/2}hg^{-1/2}$. If the principal curvatures lie in a compact
subset of $\Gamma_+$, then all derivatives of $F$ are uniformly
bounded and $(\dot F^{ij})$ is uniformly positive definite.

Suppose now that $p$ lies in the interior of the domain $\Omega$
enclosed by $M$. The support function with respect to $p$ is
\begin{equation*}
 u=\overline g(\sin r\,\partial_r,\nu).
\end{equation*}
The conformal identity above and the Codazzi equation give the standard formulas \cite{CGLS}*{Lemma 2.7}
\begin{equation}\label{eq:supportHessians}
\begin{aligned}
 \nabla_i u&=\sin r\,h^j{}_i\nabla_j r,\\
 \nabla_i\nabla_j u
 &=\sin r\,\nabla^k r\nabla_k h_{ij}
                   +\cos r\,h_{ij}-u(h^2)_{ij}.
\end{aligned}
\end{equation}
An interior centre need not satisfy $r<\pi/2$ on all of $M$. We retain
the full coordinate range $0<r<\pi$. The estimate in
Section~\ref{sec:curvature} uses only $\cos r\leq1$. 

We shall use the standard support function estimate
\cite{CabezasScheuer}*{Lemma~2.1}, cf.
\cite{GerhardtSphere}*{Section~6}. If
$B_\rho(p)\subset\Omega$, with $0<\rho<\pi/2$, then
\begin{equation}\label{eq:supportBound}
 u\geq\sin\rho,
 \qquad
 \rho\leq r\leq\pi-\rho
 \quad\text{on }M.
\end{equation}
The second pair of inequalities follows from $u\leq\sin r$ and
$0<r<\pi$.

We shall also use the rigidity theorem of Koh \cite{Koh}, in the form recalled in \cite{KohLee}*{Theorem~1}. A smooth connected closed hypersurface embedded in an open hemisphere is a geodesic sphere if $E_{k+1}/E_k$ is constant and $E_k>0$ everywhere, for some $0\leq k\leq n-1$. This includes $k=0$, with $E_0=1$, and applies
in particular to the boundary of a strictly convex domain.

\subsection{Spherical quermassintegrals}\label{subsec:quermass}
We use $\mathcal A_j$ and $f_j$ as defined in
\eqref{eq:Adefinition} and \eqref{eq:ballfunctions}. Under a normal variation $\partial_tX=\eta\nu$, their first variations are
\cite{CGLS}*{Proposition 2.9}
\begin{equation}\label{eq:Avariation}
 \frac{d}{dt}\mathcal A_{-1}(\Omega_t)=\int_{M_t}\eta\,d\mu_t,
 \qquad
 \frac{d}{dt}\mathcal A_m(\Omega_t)=(m+1)\int_{M_t}\eta\sigma_{m+1}\,d\mu_t
 \quad(0\leq m\leq n-1).
\end{equation}
To write the dimensional constants uniformly, put
\begin{equation*}
 d_0=1,\qquad d_j=j\binom nj\quad(1\leq j\leq n).
\end{equation*}
Thus $\frac{d}{dt}\mathcal A_{j-1}(\Omega_t)=d_j\int_{M_t}\eta E_j\,d\mu_t$ for
$0\leq j\leq n$. Since $\mathcal W=\cot r\,\operatorname{Id}$ on a geodesic sphere,
\begin{equation}\label{eq:ball}
\begin{aligned}
 f_{-1}(r)&=|\mathbb S^n|\int_0^r(\sin \rho)^n\,d\rho,\\
 f_m'(r)&=d_{m+1}|\mathbb S^n|\sin^{n-m-1}r\cos^{m+1}r
                    &&(0\leq m\leq n-1),\\
 f_m(0)&=0&&(-1\leq m\leq n-1).
\end{aligned}
\end{equation}
In particular $f_m^{-1}$ is defined on $(0,f_m(\pi/2))$.
The endpoint value is understood as a limit from below.

The normalization \eqref{eq:Adefinition} differs from the
integral-geometric normalization only by positive dimensional factors. More precisely, $\mathcal A_m$ is $(n+1)\binom nm$ times the integral-geometric quermassintegral of index $m+1$ for $m\geq0$, while the volume is unchanged. The endpoint identity consequently is
\begin{equation*}
 \mathcal A_n(\Omega)=|\mathbb S^n|.
\end{equation*}
For the integral-geometric definition and its relation to curvature integrals, see Solanes \cite{SolanesThesis}*{Section 3.2} and Cabezas-Rivas and Scheuer \cite{CabezasScheuer}*{Section 2}.

\begin{lemma}\label{lem:inclusion}
For $-1\leq m\leq n-1$, $\mathcal A_m$ is monotone under inclusion of spherical convex bodies contained in an open hemisphere. Every smooth weakly convex domain satisfies
\begin{equation*}
 0<\mathcal A_m(\Omega)<f_m(\pi/2).
\end{equation*}
\end{lemma}
\begin{proof}
The inclusion monotonicity follows from the integral-geometric definition, see \cite{CabezasScheuer}*{Section 2}. Choose $B_\rho(p)\subset\Omega\subset B_R(z)$ with
$0<\rho\leq R<\pi/2$. Then \eqref{eq:ball} gives
\begin{equation*}
 0<f_m(\rho)\leq\mathcal A_m(\Omega)\leq f_m(R)<f_m(\pi/2).
 \qedhere
\end{equation*}
\end{proof}

Denote the spherical inradius and circumradius by $\rho_-$ and $\rho_+$. We use \cite{CabezasScheuer}*{Proposition 3.1}. It states that if $\kappa_{\max}\leq C_0\kappa_{\min}$, then $\rho_+\leq C_1(n,C_0)\rho_-$. Its consequence for a fixed
quermassintegral is the following result.
\begin{lemma}[{\cite{CabezasScheuer}*{Corollary 3.2}}]\label{lem:scale}
Suppose $\kappa_{\max}\leq C_0\kappa_{\min}$ and
$\mathcal A_m(\Omega)=a$ for some $-1\leq m\leq n-1$.
If $r_a=f_m^{-1}(a)$, then
\begin{equation}\label{eq:scale}
 C_1^{-1}r_a\leq\rho_-(\Omega)\leq r_a<\pi/2.
\end{equation}
\end{lemma}
The dimensional change of normalization cancels in the inverse ball
function. The upper bound in \eqref{eq:scale} concerns the inradius,
not the circumradius.

\subsection{Spherical polarity}\label{subsec:polarity}
In this subsection, $\Omega$ is a smooth strictly convex domain.
We use the spherical polarity of Gerhardt
\cite{Gerhardt}*{Section~9.2}, recalled by Hu and Li
\cite{HuLi2023}*{Section~2.2}. Identify $\mathbb S^{n+1}$ with the
unit sphere in $\mathbb R^{n+2}$, and write $x\cdot y$ for the
Euclidean scalar product. The polar body is
\begin{equation}\label{eq:polarBody}
 \Omega^*
 =\{y\in\mathbb S^{n+1}:x\cdot y\leq0
                         \text{ for all }x\in\Omega\}.
\end{equation}
If $X$ parametrizes $M=\partial\Omega$ with outward unit normal
$\nu$, then the Gauss map parametrizes $M^*=\partial\Omega^*$ by
\begin{equation}\label{eq:polarGauss}
 Y=\nu,\qquad N=X,
\end{equation}
where $N$ is the outward unit normal of $M^*$. Indeed,
$\Omega^*\subset\{y:y\cdot X\leq0\}$ and $Y\cdot X=0$.
The Weingarten equation gives
\begin{equation*}
 dY=\mathcal W\,dX,\qquad dN=dX,\qquad
 \mathcal W^*=\mathcal W^{-1},\qquad
 d\mu^*=E_n(\mathcal W)\,d\mu.
\end{equation*}
Here the endomorphisms are identified by the Gauss map, as in
\cite{HuLi2023}*{Theorem~2.1}. Consequently,
\begin{equation}\label{eq:polarE}
 E_j(\mathcal W)\,d\mu
 =E_{n-j}(\mathcal W^*)\,d\mu^*,\qquad 0\leq j\leq n.
\end{equation}
Polarity reverses inclusion, satisfies $(\Omega^*)^*=\Omega$, and
sends a closed geodesic ball to
\begin{equation*}
 \bigl(\overline{B_R(p)}\bigr)^*
 =\overline{B_{\pi/2-R}(-p)}
 \qquad(0<R<\pi/2).
\end{equation*}
Thus
\begin{equation*}
 \overline{B_\rho(z)}\subset\Omega^*
 \quad\Longleftrightarrow\quad
 \Omega\subset\overline{B_{\pi/2-\rho}(-z)}.
\end{equation*}
Taking optimal radii gives
\begin{equation}\label{eq:radiusdual}
 \rho_-(\Omega^*)=\pi/2-\rho_+(\Omega).
\end{equation}

For the radial quantities, suppose that $p\in\operatorname{int}\Omega$
and $\Omega\subset B_{\pi/2}(p)$, and put $p^*=-p$. Then
$p^*\in\operatorname{int}\Omega^*$ and
$\Omega^*\subset B_{\pi/2}(p^*)$. Measure $r$ and $u$ from $p$, and
$r^*$ and $u^*$ from $p^*$. Since $u=-p\cdot\nu$ and
$u^*=-p^*\cdot N$, \eqref{eq:polarGauss} gives
\begin{equation}\label{eq:polarSupport}
 \cos r^*=p^*\cdot Y=u,
 \qquad
 u^*=-p^*\cdot N=\cos r.
\end{equation}

\subsection{Preservation and monotonicity}\label{subsec:monotonicity}
We record the normalization identities on any time interval on which
the solution of \eqref{eq:mainflow} is smooth and strictly convex. Define the time-dependent
weighted measure on $M_t$ by
\begin{equation*}
 d\omega_t=E_k(\mathcal W_t)\,d\mu_t.
\end{equation*}
Then
\begin{equation*}
 \phi(t)=\frac{\displaystyle\int_{M_t}F\,d\omega_t}
              {\displaystyle\int_{M_t}F^{-1}\,d\omega_t}.
\end{equation*}

\begin{proposition}\label{prop:moment}
Along any smooth strictly convex solution of \eqref{eq:mainflow},
\begin{equation}\label{eq:monotonicity}
\begin{aligned}
 \frac{d}{dt}\mathcal A_{k-1}(\Omega_t)&=0,\\
 \frac{d}{dt}\mathcal A_k(\Omega_t)
 &=-d_{k+1}\int_{M_t}
 \frac{F+\sqrt{\phi(t)}}{F}
 \bigl(F-\sqrt{\phi(t)}\bigr)^2\,d\omega_t
 \leq0.
\end{aligned}
\end{equation}
Moreover, $\frac{d}{dt}\mathcal A_k(\Omega_t)=0$ if and only if
$M_t$ is a geodesic sphere.
\end{proposition}

\begin{proof}
Let $\eta=\phi/F-F$. By the definition of $\phi$,
\begin{equation*}
 \int_{M_t}\eta\,d\omega_t
 =\phi(t)\int_{M_t}F^{-1}\,d\omega_t
  -\int_{M_t}F\,d\omega_t=0.
\end{equation*}
The first variation formula \eqref{eq:Avariation} therefore gives
\begin{equation*}
 \frac{d}{dt}\mathcal A_{k-1}(\Omega_t)=0
\end{equation*}
and
\begin{equation*}
\begin{aligned}
 \frac{d}{dt}\mathcal A_k(\Omega_t)
 &=d_{k+1}\int_{M_t}F\eta\,d\omega_t\\
 &=d_{k+1}\int_{M_t}\bigl(F-\sqrt{\phi(t)}\bigr)
                      \left(\frac{\phi(t)}F-F\right)\,d\omega_t\\
 &=-d_{k+1}\int_{M_t}\frac{F+\sqrt{\phi(t)}}{F}
                      \bigl(F-\sqrt{\phi(t)}\bigr)^2\,d\omega_t.
\end{aligned}
\end{equation*}
All weights are positive, so equality holds if and only if
$F=\sqrt{\phi(t)}$ on $M_t$. Since $M_t$ is closed and embedded in
an open hemisphere, with $E_k>0$, the rigidity theorem of Koh \cite{Koh}, in the form recalled in \cite{KohLee}*{Theorem~1} implies that $M_t$ is a geodesic sphere. Conversely, on a geodesic sphere $F$ is constant and the normalization gives $\phi(t)=F^2$, so $\frac{d}{dt}\mathcal A_k(\Omega_t)=0$.
\end{proof}

\subsection{Evolution equations}\label{subsec:evolution}
We rewrite the flow \eqref{eq:mainflow} as
\begin{equation}\label{eq:speed}
 \partial_tX=-G\nu,\qquad G=F-\frac{\phi(t)}F,
 \qquad F=\frac{E_{k+1}}{E_k}.
\end{equation}
The global term $\phi$ is defined in \eqref{eq:mainflow}. In the spatial
calculations below it is held fixed. By the chain rule,
\begin{equation}\label{eq:compositeDerivatives}
\begin{aligned}
 \dot G^{ij}&=\left(1+\frac\phi{F^2}\right)\dot F^{ij},\\
 \ddot G^{ij,k\ell}
 &=\left(1+\frac\phi{F^2}\right)\ddot F^{ij,k\ell}
                         -\frac{2\phi}{F^3}\dot F^{ij}\dot F^{k\ell}.
\end{aligned}
\end{equation}
The corresponding linearized spatial operator is
\begin{equation}\label{eq:linearizedOperator}
 \mathcal L=\dot G^{ij}\nabla_i\nabla_j
     =\left(1+\frac\phi{F^2}\right)\dot F^{ij}\nabla_i\nabla_j.
\end{equation}
This is the local spatial operator, not the full linearization of the
nonlocal equation. The following identities use only that $\phi$ is
positive and constant on each time slice. Thus they also apply to the
polar equation after the substitutions made in
Section~\ref{subsec:polarflow}.

\begin{lemma}
\label{lem:evolution}
Along \eqref{eq:speed}, the metric and area element satisfy
\begin{equation}\label{eq:variations}
 \partial_tg_{ij}=-2Gh_{ij},
 \qquad
 \partial_tg^{ij}=2Gh^{ij},
 \qquad
 \partial_t(d\mu_t)=-GH\,d\mu_t.
\end{equation}
With $\mathcal L$ as in \eqref{eq:linearizedOperator}, we have
\begin{equation}\label{eq:Bevol}
\begin{aligned}
 (\partial_t-\mathcal L)h^i{}_j={}&
 \ddot G^{k\ell,pq}\nabla^i h_{k\ell}\nabla_j h_{pq}\\
 &+\left(1+\frac{\phi}{F^2}\right)
 \dot F^{k\ell}\bigl((h^2)_{k\ell}-g_{k\ell}\bigr)h^i{}_j
 -\frac{2\phi}{F}h^i{}_kh^k{}_j
 +2F\delta^i_j,
\end{aligned}
\end{equation}
\begin{equation}\label{eq:hevol}
\begin{aligned}
 (\partial_t-\mathcal L)h_{ij}={}&
 \ddot G^{k\ell,pq}\nabla_i h_{k\ell}\nabla_j h_{pq}\\
 &+\left(1+\frac{\phi}{F^2}\right)
 \dot F^{k\ell}\bigl((h^2)_{k\ell}-g_{k\ell}\bigr)h_{ij}
 -2F(h^2)_{ij}
 +2Fg_{ij},
\end{aligned}
\end{equation}
and
\begin{equation}\label{eq:HpinchEvolution}
\begin{aligned}
 (\partial_t-\mathcal L)H={}&
 \ddot G^{ij,k\ell}\nabla_p h_{ij}\nabla^p h_{k\ell}\\
 &+\left(1+\frac{\phi}{F^2}\right)
 \dot F^{ij}\bigl((h^2)_{ij}-g_{ij}\bigr)H
 -\frac{2\phi}{F}|A|^2
 +2nF.
\end{aligned}
\end{equation}
Furthermore,
\begin{equation}\label{eq:Fscalar}
 (\partial_t-\mathcal L)F
 =
 -\frac{2\phi}{F^3}
 \dot F^{ij}\nabla_iF\nabla_jF
 +
 G\dot F^{ij}\bigl((h^2)_{ij}+g_{ij}\bigr),
\end{equation}
\begin{equation}\label{eq:reciprocal}
 (\partial_t-\mathcal L)F^{-1}
 =
 -\frac{2}{F^3}
 \dot F^{ij}\nabla_iF\nabla_jF
 -
 \frac{G}{F^2}
 \dot F^{ij}\bigl((h^2)_{ij}+g_{ij}\bigr),
\end{equation}
and the support function about a fixed point satisfies
\begin{equation}\label{eq:supportEvolution}
 (\partial_t-\mathcal L)u
 =
 -2F\cos r
 +
 u\left(1+\frac{\phi}{F^2}\right)
 \dot F^{ij}(h^2)_{ij}.
\end{equation}
For $1\leq j\leq n$,
\begin{equation}\label{eq:Ejvariation}
 \partial_tE_j
 =
 \dot E_j^{k\ell}\nabla_k\nabla_\ell G
 +
 G\bigl(HE_j-(n-j)E_{j+1}+jE_{j-1}\bigr).
\end{equation}
\end{lemma}

\begin{proof}
The first variation formulas in \cite{CGLS}*{Proposition~2.8}, with
normal speed $-G$, give \eqref{eq:variations} and
\begin{equation}\label{eq:uncommutedEvolution}
 \partial_t h^i{}_j
 =
 \nabla^i\nabla_jG
 +
 G\bigl(h^i{}_kh^k{}_j+\delta^i_j\bigr).
\end{equation}
The contracted Simons identity
\cite{CGLS}*{Lemma~2.12} gives
\begin{equation}\label{eq:Simons}
\begin{aligned}
 \nabla_i\nabla_jF={}&
 \dot F^{k\ell}\nabla_k\nabla_\ell h_{ij}
 +
 \ddot F^{k\ell,pq}
 \nabla_i h_{k\ell}\nabla_j h_{pq}\\
 &+
 \dot F^{k\ell}\bigl((h^2)_{k\ell}-g_{k\ell}\bigr)h_{ij}
 -
 F(h^2)_{ij}
 +
 Fg_{ij}.
\end{aligned}
\end{equation}
Using \eqref{eq:compositeDerivatives} in
\eqref{eq:uncommutedEvolution} and then
\eqref{eq:Simons}, we obtain \eqref{eq:Bevol}. Lowering the index and using
\eqref{eq:variations} gives \eqref{eq:hevol}.
Tracing \eqref{eq:Bevol} gives \eqref{eq:HpinchEvolution}.

Applying the chain rule to \eqref{eq:uncommutedEvolution},
\begin{equation*}
\begin{aligned}
 \partial_tF
 &=
 \dot F^{ij}\nabla_i\nabla_jG
 +
 G\dot F^{ij}\bigl((h^2)_{ij}+g_{ij}\bigr)\\
 &=
 \mathcal LF
 -
 \frac{2\phi}{F^3}
 \dot F^{ij}\nabla_iF\nabla_jF
 +
 G\dot F^{ij}\bigl((h^2)_{ij}+g_{ij}\bigr),
\end{aligned}
\end{equation*}
which proves \eqref{eq:Fscalar}. Since
\begin{equation*}
 \mathcal L(F^{-1})
 =
 -F^{-2}\mathcal LF
 +
 2F^{-3}
 \left(1+\frac{\phi}{F^2}\right)
 \dot F^{ij}\nabla_iF\nabla_jF,
\end{equation*}
\eqref{eq:reciprocal} follows directly.

For the support function, using
$\overline\nabla_{\partial_t}\nu=\nabla G$ and
\eqref{equ-conformal}, we obtain
\begin{equation*}
 \partial_tu
 =
 -G\cos r
 +
 \sin r\,g(\nabla r,\nabla G).
\end{equation*}
On the other hand, \eqref{eq:supportHessians} and Euler's identity give
\begin{equation*}
 \mathcal Lu
 =
 \left(1+\frac{\phi}{F^2}\right)
 \left(
 \sin r\,g(\nabla r,\nabla F)
 +
 F\cos r
 -
 u\dot F^{ij}(h^2)_{ij}
 \right).
\end{equation*}
Since
\begin{equation*}
 \nabla G
 =
 \left(1+\frac{\phi}{F^2}\right)\nabla F,
\end{equation*}
subtraction gives \eqref{eq:supportEvolution}.

Finally, \eqref{eq:Ejvariation} follows from
\eqref{eq:uncommutedEvolution}, the chain rule, and
\eqref{eq:Ejidentities}.
\end{proof}

\section{Curvature pinching estimates}\label{sec:flow}
In this section we prove that strict convexity is preserved along
\eqref{eq:speed} and establish a curvature pinching estimate which
improves exponentially in time.

Let $M_t$, $t\in[0,T_{\max})$, denote the maximal smooth parabolic
solution of \eqref{eq:speed} given by the short-time existence theory
recalled in Section~\ref{subsec:main}.

\subsection{Lower test functions for the least principal curvature}
\label{subsec:eigenvalue}

We first recall the lower test function lemma of Brendle, Choi, and Daskalopoulos \cite{BCD}*{Lemma 5} for the least principal curvature.

\begin{lemma}\label{lem:BCD}
Suppose that $\varphi$ is a smooth function on $M$ such that
$\kappa_1\geq\varphi$ everywhere and $\kappa_1=\varphi$ at $\bar p$.
Let $\mu$ denote the multiplicity of the smallest principal curvature at
$\bar p$, so that
\begin{equation*}
 \kappa_1=\cdots=\kappa_\mu<\kappa_{\mu+1}\leq\cdots\leq\kappa_n.
\end{equation*}
The sums over $j>\mu$ below are understood to be empty if $\mu=n$. Choose an orthonormal principal frame at $\bar p$. Then, at $\bar p$,
\begin{equation}\label{eq:BCDfirst}
 \nabla_i h_{j\ell}=\nabla_i\varphi\,\delta_{j\ell},\qquad 1\leq i\leq n,\ 1\leq j,\ell\leq\mu.
\end{equation}
Moreover, for each $i$,
\begin{equation}\label{eq:BCDsecond}
 \nabla_i\nabla_i\varphi\leq\nabla_i\nabla_i h_{11}
 -2\sum_{j>\mu}\frac{(\nabla_i h_{1j})^2}{\kappa_j-\kappa_1}.
\end{equation}
\end{lemma}

We next derive the corresponding time comparison.
\begin{lemma}\label{lem:parabolicContact}
	Let $M_t$, $t\in[0,T_{\max})$, be a smooth solution
	of \eqref{eq:speed}. Suppose that a smooth function $\psi$
	satisfies
	\begin{equation*}
	\psi\leq\kappa_1
	\quad\text{on }U\times(t_0-\eta,t_0],
	\qquad
	\psi(x_0,t_0)=\kappa_1(x_0,t_0),
	\end{equation*}
	where $U$ is a neighborhood of $x_0$ and
	$0<\eta<t_0<T_{\max}$.
	Choose an orthonormal principal frame at $(x_0,t_0)$, and let
	$\mu$ denote the multiplicity of $\kappa_1$ there, so that
	\begin{equation*}
	\kappa_1=\cdots=\kappa_\mu,
	\qquad
	\kappa_j>\kappa_1\quad(j>\mu).
	\end{equation*}
	Then, at $(x_0,t_0)$,
	\begin{equation}\label{eq:eigenvalueContact}
		\begin{aligned}
			(\partial_t-\mathcal L)\psi\geq{}&
			\ddot G^{ij,k\ell}\nabla_1h_{ij}\nabla_1h_{k\ell}
			+2\sum_{i=1}^n\sum_{j>\mu}
			\frac{\dot G^i(\nabla_i h_{1j})^2}{\kappa_j-\kappa_1}\\
			&+\left(1+\frac{\phi}{F^2}\right)
			\dot F^{ij}\bigl((h^2)_{ij}-g_{ij}\bigr)\kappa_1
			-\frac{2\phi}{F}\kappa_1^2+2F,
		\end{aligned}
	\end{equation}
	where $\mathcal L$ is defined by \eqref{eq:linearizedOperator}.
	The sum over $j>\mu$ is understood to be empty if $\mu=n$.
\end{lemma}

\begin{proof}
	At $x_0$, keep the tangent vector $v=e_1\in T_{x_0}M$ fixed
	as $t$ varies. For $t\leq t_0$ sufficiently close to $t_0$,
	\begin{equation*}
	\frac{h(v,v)}{g(v,v)}-\psi(x_0,t)
	\geq \kappa_1(x_0,t)-\psi(x_0,t)\geq0,
	\end{equation*}
	with equality at $t=t_0$. The left time derivative of this
	expression at $t_0$ is therefore nonpositive.
	Using \eqref{eq:variations}, we obtain
	\begin{equation}\label{eq:eigenvalueTimeContact}
		\partial_t\psi
		\geq (\partial_t h)_{11}
		-\kappa_1(\partial_tg)_{11}
		=(\partial_t h)_{11}+2G\kappa_1^2
	\end{equation}
	at $(x_0,t_0)$.
	
	Apply Lemma~\ref{lem:BCD} on the time slice $t=t_0$ to
	$\varphi=\psi(\cdot,t_0)$. Since
	$\dot G^{ij}=\dot G^i\delta_{ij}$ with $\dot G^i>0$
	in the chosen frame, we obtain
	\begin{equation*}
	\mathcal L\psi
	\leq (\mathcal Lh)_{11}
	-2\sum_{i=1}^n\sum_{j>\mu}
	\frac{\dot G^i(\nabla_i h_{1j})^2}{\kappa_j-\kappa_1}.
	\end{equation*}
	Combining this with \eqref{eq:eigenvalueTimeContact} and
	\eqref{eq:hevol}, and using $G=F-\phi/F$, proves
	\eqref{eq:eigenvalueContact}.
\end{proof}

\subsection{Preservation of strict convexity}

\begin{lemma}
\label{prop:convexity}
If $M_0$ is strictly convex, then along \eqref{eq:speed} the evolving hypersurface $M_t$ remains strictly convex for
all $t\in[0,T_{\max})$.
\end{lemma}

\begin{proof}
Fix $T<T_{\max}$ and choose $0<\delta_0<\min_{M_0}\kappa_1$. Since the solution is smooth and parabolic on $[0,T]$, $F$ has a
positive lower bound and $\phi$ is bounded above on this interval. Set
\begin{equation*}
 C_T=1+\max_{M\times[0,T]}
       \left\{(k+1)\left(1+\frac \phi{F^2}\right)
                              +\frac{2\phi\delta_0}F\right\},
 \qquad \delta(t)=\delta_0e^{-C_Tt}.
\end{equation*}

Suppose that $\kappa_1\geq\delta$ does not hold on $M\times[0,T]$. Since $\kappa_1(\cdot,0)>\delta_0$, there is a first contact point $(x_0,t_0)$ with $t_0>0$ such that
\begin{equation*}
 \kappa_1(x_0,t_0)=\delta(t_0).
\end{equation*}
Choose an orthonormal principal frame at $(x_0,t_0)$ and let $\mu$ be the multiplicity of $\kappa_1$. Thus
\begin{equation*}
 0<\kappa_1=\cdots=\kappa_\mu=\delta
 <\kappa_{\mu+1}\leq\cdots\leq\kappa_n.
\end{equation*}
In particular, $\mathcal W(x_0,t_0)\in\Gamma_+$.

Applying Lemma~\ref{lem:BCD} on the time slice $t=t_0$
with $\varphi\equiv\delta(t_0)$, we obtain from
\eqref{eq:BCDfirst} that
\begin{equation*}
    \nabla_i h_{j\ell}=0 \qquad (1\leq i\leq n,\ 1\leq j,\ell\leq\mu).
\end{equation*}
The Codazzi equation then gives
\begin{equation}\label{equ-dihjl}
 \nabla_i h_{1j}=\nabla_1h_{ij}=0
 \qquad(1\leq i\leq n,\ 1\leq j\leq\mu).
\end{equation}
At the contact time $t_0$, $G=G_{\phi(t_0)}$. By
Lemma~\ref{lem:Gconcavity}, $G$ is strictly increasing, concave, and inverse-concave. Hence \eqref{eq:inverseConcavity} applies with $\Phi=G$. Moreover, using $\kappa_j-\delta<\kappa_j$ for $j>\mu$ and \eqref{equ-dihjl}, we obtain

\begin{align}
 &\ddot G^{ij,k\ell}\nabla_1h_{ij}\nabla_1h_{k\ell}
       +2\sum_i\sum_{j>\mu}
          \frac{\dot G^i(\nabla_i h_{1j})^2}{\kappa_j-\delta} \nonumber\\
 \geq&\ddot G^{ij,k\ell}\nabla_1h_{ij}\nabla_1h_{k\ell}
       +2\sum_{i,j}\frac{\dot G^i}{\kappa_j}
                                  (\nabla_1h_{ij})^2\geq0. \label{eq:convexityGradient}
\end{align}

We now apply Lemma~\ref{lem:parabolicContact} with
$\psi(x,t)=\delta(t)$, which touches $\kappa_1$ from below
at $(x_0,t_0)$ on the preceding space-time neighborhood.
Since $(\partial_t-\mathcal L)\delta=-C_T\delta$,
\eqref{eq:eigenvalueContact} and
\eqref{eq:convexityGradient} give

\begin{equation*}
 -C_T\delta \geq \delta\left(1+\frac{\phi}{F^2}\right)
 \dot F^{ij}\bigl((h^2)_{ij}-g_{ij}\bigr)
 -\frac{2\phi\delta^2}{F}
 +2F
\end{equation*}
at $(x_0,t_0)$. Using \eqref{eq:quotientTraces}, $0<\delta\leq\delta_0$, we obtain
\begin{align*}
 0
 &\geq
 \delta
 \left\{
 C_T
 -(k+1)\left(1+\frac{\phi}{F^2}\right)
 -\frac{2\phi\delta_0}{F}
 \right\}
 +2F\\
 &\geq
 \delta+2F>0,
\end{align*}
which is impossible. Hence no first contact occurs, and
\begin{equation*}
 \kappa_1(\cdot,t)\geq 
 \delta(t)=\delta_0e^{-C_Tt}>0
 \qquad(0\leq t\leq T).
\end{equation*}
Since $T<T_{\max}$ is arbitrary, the solution remains strictly convex
throughout $[0,T_{\max})$. By the global convexity theorem of
do Carmo and Warner \cite{DoCarmoWarner}*{Theorem~1.1},
each $M_t$ is embedded and bounds a strictly convex domain
$\Omega_t$ contained in an open hemisphere.
\end{proof}

\subsection{Curvature pinching estimates}

\begin{lemma}[Pinching]
\label{lem:pinch}
Let
\begin{equation*}
 0<\varepsilon_0<
 \min_{M_0}\frac{\kappa_1}{H}.
\end{equation*}
Then along \eqref{eq:speed},
\begin{equation}\label{eq:pinch}
 h_{ij}\geq\varepsilon_0 H g_{ij}
 \qquad\text{for all }t\in[0,T_{\max}).
\end{equation}
In particular,
\begin{equation}\label{eq:curvatureRatio}
 \frac{\kappa_{\max}}{\kappa_{\min}}
 \leq\varepsilon_0^{-1}.
\end{equation}
\end{lemma}

\begin{proof}
By Lemma~\ref{prop:convexity}, the solution remains strictly
convex. Since $H\geq n\kappa_1$, the choice of $\varepsilon_0$ implies $0<\varepsilon_0<1/n$.

We first derive a consequence of
Lemma~\ref{lem:parabolicContact} that will also be used
in Lemma~\ref{lem:improvedPinch}.
Let $\varepsilon(t)\in(0,1/n)$ be smooth, and suppose that
\begin{equation*}
 \kappa_1\geq\varepsilon(t)H
\end{equation*}
holds up to a first contact at $(x_0,t_0)$, where $t_0>0$. Choose an
orthonormal principal frame at this point, and let $\mu$ be the
multiplicity of the least principal curvature. Thus
\begin{equation*}
 \kappa_1=\cdots=\kappa_\mu=\varepsilon H,
 \qquad
 \kappa_j>\kappa_1\quad(j>\mu).
\end{equation*}
Applying Lemma~\ref{lem:BCD} on the time slice $t=t_0$
with $\varphi=\varepsilon(t_0)H(\cdot,t_0)$, we obtain
from \eqref{eq:BCDfirst} that
\begin{equation}\label{eq:pinchNullData}
 \nabla_p h_{ij}=\varepsilon\nabla_pH\,\delta_{ij}
 \qquad(1\leq p \leq n,\ 1\leq i,j\leq\mu).
\end{equation}
In particular,
\begin{equation*}
 \nabla_p h_{11}=\varepsilon\nabla_pH
 =\varepsilon\sum_q\nabla_p h_{qq},
 \qquad
 \nabla_p h_{1j}=0\quad(2\leq j\leq\mu).
\end{equation*}

We next apply Lemma~\ref{lem:parabolicContact} with
$\psi(x,t)=\varepsilon(t)H(x,t)$.
Since
\begin{equation*}
 (\partial_t-\mathcal L)(\varepsilon H)
 =\varepsilon(\partial_t-\mathcal L)H+ \frac{d}{dt}\varepsilon(t) H,
\end{equation*}
\eqref{eq:eigenvalueContact} yields
\begin{align}
 \varepsilon(\partial_t-\mathcal L)H + \frac{d}{dt}\varepsilon(t) H
 \geq&
 \ddot G^{ij,k\ell}\nabla_1h_{ij}\nabla_1h_{k\ell}
 +2\sum_i\sum_{j>\mu}
   \frac{\dot G^i(\nabla_i h_{1j})^2}{\kappa_j-\kappa_1}   \nonumber\\
 &+\left(1+\frac{\phi}{F^2}\right)
   \dot F^{ij}\bigl((h^2)_{ij}-g_{ij}\bigr)\varepsilon H
 -\frac{2\phi}{F}\varepsilon^2H^2+2F.  \label{equ-1}
\end{align}
On the other hand, \eqref{eq:HpinchEvolution} gives
\begin{align}
 \varepsilon(\partial_t-\mathcal L)H
 ={}&
 \varepsilon\ddot G^{ij,k\ell}
       \nabla_p h_{ij}\nabla^p h_{k\ell}   \nonumber\\
 &+\varepsilon\left(1+\frac{\phi}{F^2}\right)
   \dot F^{ij}\bigl((h^2)_{ij}-g_{ij}\bigr)H
 -\frac{2\phi\varepsilon}{F}|A|^2+2n\varepsilon F.  \label{equ-2}
\end{align}
Substituting \eqref{equ-2} into \eqref{equ-1} cancels the terms linear in $H$ and gives
\begin{equation}\label{eq:pinchingContact}
\begin{aligned}
 0\geq{}&
 \ddot G^{ij,k\ell}
 \left(
   \nabla_1h_{ij}\nabla_1h_{k\ell}
   -\varepsilon\sum_p\nabla_ph_{ij}\nabla_ph_{k\ell}
 \right)\\
 &+2\sum_i\sum_{j>\mu}
   \frac{\dot G^i(\nabla_i h_{1j})^2}{\kappa_j-\kappa_1}\\
 &+\frac{2\phi\varepsilon}{F}
      \bigl(|A|^2-\varepsilon H^2\bigr)
   +2F(1-n\varepsilon)-\frac{d}{dt}\varepsilon(t) H.
\end{aligned}
\end{equation}

By the Codazzi equation,
\begin{equation*}
 T_{pij}:=\nabla_p h_{ij}
\end{equation*}
is a totally symmetric three-tensor, and \eqref{eq:pinchNullData} is precisely
\eqref{eq:algebraicNullData}. At the contact time $t_0$, $G=G_{\phi(t_0)}$. By Lemma~\ref{lem:Gconcavity}, $G$ is strictly increasing, concave, and inverse-concave. Applying
Lemma~\ref{lem:AndrewsGradient} with $\Phi=G$, we conclude from
\eqref{eq:pinchingGradient} that the sum of the first two terms on the right-hand side of \eqref{eq:pinchingContact} is nonnegative. Consequently, at the first contact point $(x_0,t_0)$,
\begin{equation}\label{eq:pinchingReaction}
 0\geq
 \frac{2\phi\varepsilon}{F}
 \bigl(|A|^2-\varepsilon H^2\bigr)
 +2F(1-n\varepsilon)- \frac{d}{dt}\varepsilon(t)H.
\end{equation}

We now take $\varepsilon(t)\equiv\varepsilon_0$. Then
$\frac{d}{dt}\varepsilon(t)=0$. Since $|A|^2\geq H^2/n$,
\begin{equation*}
 |A|^2-\varepsilon_0H^2
 \geq\left(\frac1n-\varepsilon_0\right)H^2>0,
\end{equation*}
and also
\begin{equation*}
 2F(1-n\varepsilon_0)>0.
\end{equation*}
Since $F>0$ and $\phi>0$, these inequalities make the
right-hand side of \eqref{eq:pinchingReaction} strictly
positive, a contradiction. Hence \eqref{eq:pinch} holds
throughout the smooth existence interval.

Finally,
\begin{equation*}
 \kappa_{\min}\geq\varepsilon_0H,
 \qquad
 \kappa_{\max}\leq H,
\end{equation*}
and therefore \eqref{eq:curvatureRatio} holds.
\end{proof}

The positive ambient-curvature term in \eqref{eq:pinchingContact}
also improves this ratio with time.

\begin{lemma}[Improved pinching]\label{lem:improvedPinch}
Let $\varepsilon_0$ be as in Lemma~\ref{lem:pinch}. For any
$0<\gamma\leq2n\varepsilon_0$, the solution along \eqref{eq:speed} satisfies
\begin{equation}\label{eq:improve}
 h_{ij}\geq
 \left[
 \frac1n-
 \left(\frac1n-\varepsilon_0\right)e^{-\gamma t}
 \right]Hg_{ij}
 \qquad\text{for all }t\in[0,T_{\max}).
\end{equation}
\end{lemma}

\begin{proof}
Set
\begin{equation*}
 \varepsilon(t)
 =
 \frac1n-\left(\frac1n-\varepsilon_0\right)e^{-\gamma t}.
\end{equation*}
Suppose that the asserted estimate fails, and let $(x_0,t_0)$ be a
first contact point, so that $\kappa_1=\varepsilon H$ at $(x_0,t_0)$. The inequality \eqref{eq:pinchingReaction}, established
for a general time-dependent $\varepsilon$, applies
at $(x_0,t_0)$.

Since
\begin{equation*}
 \varepsilon_0\leq\varepsilon (t)<\frac1n,
 \qquad
 \frac{d}{dt}\varepsilon (t)
 =
 \gamma\left(\frac1n-\varepsilon\right),
\end{equation*}
\eqref{eq:minmaxF} and Lemma~\ref{lem:pinch} give
\begin{equation*}
 F\geq\kappa_1\geq\varepsilon_0H.
\end{equation*}
Therefore
\begin{equation*}
\begin{aligned}
 2F(1-n\varepsilon)-\frac{d}{dt}\varepsilon(t) H
 &=
 (2nF-\gamma H)
 \left(\frac1n-\varepsilon\right)\\
 &\geq
 (2n\varepsilon_0-\gamma)H
 \left(\frac1n-\varepsilon\right)
 \geq0.
\end{aligned}
\end{equation*}
Moreover, using $|A|^2\geq H^2/n$,
\begin{equation*}
 \frac{2\phi\varepsilon}{F}
 \bigl(|A|^2-\varepsilon H^2\bigr)
 \geq
 \frac{2\phi\varepsilon H^2}{F}
 \left(\frac1n-\varepsilon\right)
 >0.
\end{equation*}
These inequalities contradict \eqref{eq:pinchingReaction}.
Hence \eqref{eq:improve} holds throughout the smooth
existence interval.
\end{proof}

\section{Uniform curvature estimates}\label{sec:curvature}

In this section we establish uniform two-sided curvature estimates and
the higher order estimates required for long-time existence. We first
prove an upper curvature bound by a Tso-type argument. The lower bound
then follows by applying the same estimate to the polar flow. Finally,
we derive the higher order estimates and extend the solution for all
time.

\subsection{An upper curvature bound}

The following estimate will be applied to both the original flow and the polar flow. Its proof uses only that the global coefficient is positive and spatially constant.

\begin{lemma}\label{prop:upper}
Let $M_t$, $t\in[0,T)$, be a smooth strictly convex solution of
\begin{equation*}
 \partial_tX=\left(\frac{\phi(t)}F-F\right)\nu,
 \qquad
 F=\frac{E_{k+1}}{E_k},
\end{equation*}
where $\phi(t)>0$ is spatially constant. Suppose that
$\mathcal A_m(\Omega_t)$ is preserved for some
$-1\leq m\leq n-1$. Then
\begin{equation}\label{eq:upperCurvature}
 \kappa_i\leq C,
 \qquad 1\leq i\leq n,
\end{equation}
where $C$ depends only on $M_0$, $n$, $k$, and $m$.
\end{lemma}

\begin{proof}
By Lemma~\ref{lem:pinch}, the principal curvatures have a uniform ratio
bound. Since $\mathcal A_m$ is preserved, Lemma~\ref{lem:scale} gives a
constant $r_0>0$ such that
\begin{equation*}
 \rho_-(\Omega_t)\geq r_0
\end{equation*}
throughout the smooth existence interval.

Fix $t_0$ and choose $B_{r_0}(p)\subset\Omega_{t_0}$. Let $\rho(t)$ solve
\begin{equation*}
 \frac{d}{dt}\rho=-\cot\rho,
 \qquad
 \rho(t_0)=\frac{r_0}{2}.
\end{equation*}
Since
\begin{equation*}
 \cos\rho(t)=e^{t-t_0}\cos\frac{r_0}{2},
\end{equation*}
we have $\rho(t)\geq r_0/4$ whenever
\begin{equation*}
 0\leq t-t_0\leq T_0,
 \qquad
 T_0=\log\frac{\cos(r_0/4)}{\cos(r_0/2)}>0.
\end{equation*}
The corresponding shrinking sphere remains inside $\Omega_t$ on this
interval. Indeed, at a first contact from inside,
$\mathcal W\leq\cot\rho\,\operatorname{Id}$, and hence \eqref{eq:minmaxF} gives
$F\leq\cot\rho$. Therefore
\begin{equation*}
 \frac{\phi}{F}-F>-F\geq-\cot\rho=\frac{d}{dt}\rho,
\end{equation*}
which rules out first contact. Consequently, by \eqref{eq:supportBound},
\begin{equation}\label{eq:shortsupport}
 B_{r_0/4}(p)\subset\Omega_t,
 \qquad
 u\geq\sin\frac{r_0}{4}.
\end{equation}
The center $p$ is fixed on this interval and may be chosen again when the
argument is restarted.

Following the support function argument of Tso \cite{Tso}, we set
\begin{equation*}
 d=\frac12\sin\frac{r_0}{4},
 \qquad
 Z=\frac{F}{u-d}.
\end{equation*}
By \eqref{eq:shortsupport}, $d\leq u-d\leq1$. At a spatial maximum of
$Z$,
\begin{equation*}
 \nabla F=\frac{F}{u-d}\nabla u.
\end{equation*}
Using \eqref{eq:Fscalar}, \eqref{eq:supportEvolution}, and the operator
\eqref{eq:linearizedOperator}, the gradient terms produced by
differentiating the quotient cancel, while the remaining quadratic
gradient term is nonpositive. We obtain
\begin{equation*}
\begin{aligned}
 (\partial_t-\mathcal L)Z
 &\leq
 -\frac{dF+\phi(2u-d)/F}{(u-d)^2}
       \dot F^{ij}(h^2)_{ij}\\
 &\quad
 +\frac{G}{u-d}\dot F^{ij}g_{ij}
 +\frac{2\cos r\,F^2}{(u-d)^2}.
\end{aligned}
\end{equation*}
Here we used
\begin{equation*}
 (u-d)G-uF\left(1+\frac{\phi}{F^2}\right)
 =-dF-\frac{\phi(2u-d)}F.
\end{equation*}
Since $\phi>0$, the term containing $\phi\dot F^{ij}(h^2)_{ij}$ is
nonpositive. Using $G\leq F$, $\cos r\leq1$, and
\eqref{eq:quotientTraces}, we therefore obtain, at a spatial maximum of
$Z$,
\begin{equation*}
 (\partial_t-\mathcal L)Z
 \leq-d^2Z^3+2Z^2+(k+1)Z.
\end{equation*}
Choose $Z_0>0$, depending only on $d$ and $k$, such that
\begin{equation*}
 -d^2Z^3+2Z^2+(k+1)Z
 \leq-\frac12d^2Z^3
 \qquad\text{whenever }Z\geq Z_0.
\end{equation*}
For $t>t_0$, set
\begin{equation*}
 \overline Z(t)=Z_0+\frac{2}{d\sqrt{t-t_0}}.
\end{equation*}
Since $\overline Z(t)\to+\infty$ as $t\downarrow t_0$, the maximum
principle gives
\begin{equation}\label{eq:TsoTimeBound}
 Z(\cdot,t)\leq
 Z_0+\frac{2}{d\sqrt{t-t_0}}.
\end{equation}
Indeed, at a first contact point with $Z=\overline Z>Z_0$,
\begin{align*}
 0
 &\leq(\partial_t-\mathcal L)(Z-\overline Z)\\
 &\leq-\frac12d^2\overline Z^3-\overline Z'
 \leq-\frac{3}{d(t-t_0)^{3/2}}<0,
\end{align*}
which is impossible.

For the interval starting at $t=0$, the usual maximum principle gives
\begin{equation*}
 Z\leq\max\left\{\max_{M_0}Z,Z_0\right\}.
\end{equation*}
For $t\geq T_0/2$, apply \eqref{eq:TsoTimeBound} with
$t_0=t-T_0/2$. Since $u-d\leq1$,
\begin{equation*}
 F(\cdot,t)\leq Z(\cdot,t)
 \leq Z_0+\frac{2\sqrt2}{d\sqrt{T_0}}.
\end{equation*}
Thus $F$ is uniformly bounded from above. Finally,
Lemma~\ref{lem:pinch} and \eqref{eq:minmaxF} give
\begin{equation*}
 \kappa_{\max}
 \leq H
 \leq\varepsilon_0^{-1}\kappa_{\min}
 \leq\varepsilon_0^{-1}F.
\end{equation*}
This proves the desired curvature bound.
\end{proof}

\subsection{The polar flow and the lower curvature bound}\label{subsec:polarflow}
Return to the solution of \eqref{eq:mainflow}. For each $t$, let
$M_t^*=\partial\Omega_t^*$ be the polar hypersurface. By
\eqref{eq:dualF} and the polar identities in
Section~\ref{subsec:polarity}, if $\ell=n-k$, then
\begin{equation*}
 F^*=\frac{E_\ell(\mathcal W^*)}{E_{\ell-1}(\mathcal W^*)}=\frac1F.
\end{equation*}
By \eqref{eq:polarGauss}, $Y=\nu$ and $N=X$. Differentiating
$\nu\cdot X=0$, we obtain the normal component of the polar velocity,
\begin{equation*}
 \partial_tY\cdot N
 =\partial_t\nu\cdot X
 =-\nu\cdot\partial_tX
 =F-\frac{\phi(t)}F.
\end{equation*}
We remove the tangential component by a time-dependent change of
parametrization, as in \cite{HuLi2023}*{Section~2.3}. The polar flow equations below are written in this normal parametrization.

Define
\begin{equation}\label{eq:polarTime}
 \tau(t)=\int_0^t\phi(s)\,d s,
 \qquad
 \phi^*(\tau(t))=\frac1{\phi(t)}.
\end{equation}
Since $\phi(t)>0$, this is an increasing change of time. Writing
$M_\tau^*=M_{t(\tau)}^*$ and $\Omega_\tau^*=\Omega_{t(\tau)}^*$, we have
$dt/d\tau=1/\phi(t)$ and hence
\begin{equation}\label{eq:polarflow}
 \partial_\tau Y
 =\left(\frac{\phi^*}{F^*}-F^*\right)N.
\end{equation}

We next identify the global term. At corresponding times,
\eqref{eq:polarE} gives
\begin{equation*}
 E_k(\mathcal W_t)\,d\mu_t
 =E_\ell(\mathcal W_\tau^*)\,d\mu_\tau^*.
\end{equation*}
Together with $F=1/F^*$, this yields
\begin{equation*}
 \int_{M_t}E_kF\,d\mu_t
 =\int_{M_\tau^*}\frac{E_\ell}{F^*}\,d\mu_\tau^*,
 \qquad
 \int_{M_t}\frac{E_k}{F}\,d\mu_t
 =\int_{M_\tau^*}E_\ell F^*\,d\mu_\tau^*.
\end{equation*}
Therefore
\begin{equation*}
 \phi^*(\tau)
 =\frac{\displaystyle\int_{M_\tau^*}E_\ell F^*\,d\mu_\tau^*}
        {\displaystyle\int_{M_\tau^*}E_\ell/F^*\,d\mu_\tau^*}.
\end{equation*}
By the first variation formula \eqref{eq:Avariation},
\begin{equation}\label{eq:polarPreserve}
 \frac{d}{d\tau}\mathcal A_{\ell-1}(\Omega_\tau^*)
 =d_\ell\int_{M_\tau^*}E_\ell
       \left(\frac{\phi^*}{F^*}-F^*\right)\,d\mu_\tau^*=0.
\end{equation}
Thus the polar evolution has the speed form considered in
Lemma~\ref{prop:upper}, with quotient index $\ell-1$ and preserved
quantity $\mathcal A_{\ell-1}$. These indices are independent in
that lemma.

\begin{lemma}\label{prop:two}
Along the flow \eqref{eq:mainflow}, there exists a constant $C>1$,
depending only on $M_0$, $n$, and $k$, such that
\begin{equation}\label{eq:two}
 C^{-1}\leq\kappa_i\leq C,
 \qquad
 C^{-1}\leq F\leq C,
 \qquad
 C^{-1}\leq\phi(t)\leq C
\end{equation}
for $1\leq i\leq n$. Moreover, there are constants
$r_0>0$ and $R_0<\pi/2$, depending only on $M_0$, $n$, and $k$, such that
\begin{equation}\label{eq:radiusBounds}
 \rho_-(\Omega_t),\ \rho_-(\Omega_t^*)\geq r_0,
 \qquad
 \rho_+(\Omega_t),\ \rho_+(\Omega_t^*)\leq R_0.
\end{equation}
\end{lemma}

\begin{proof}
By Proposition~\ref{prop:moment}, the original flow preserves
$\mathcal A_{k-1}(\Omega_t)$. Lemma~\ref{prop:upper} therefore gives
$\mathcal W\leq C\operatorname{Id}$, while Lemmas~\ref{lem:pinch} and
\ref{lem:scale} give a uniform positive lower bound for
$\rho_-(\Omega_t)$.

For the polar flow, \eqref{eq:polarflow} and \eqref{eq:polarPreserve}
allow us to apply Lemma~\ref{prop:upper} with the quotient index
$k$ replaced by $\ell-1$ and with $m=\ell-1$. Hence
$\mathcal W^*\leq C\operatorname{Id}$. The same application of
Lemmas~\ref{lem:pinch} and \ref{lem:scale} gives a uniform positive
lower bound for $\rho_-(\Omega_\tau^*)$. Since
$\mathcal W^*=\mathcal W^{-1}$, we conclude that
\begin{equation*}
 c\operatorname{Id}\leq\mathcal W\leq C\operatorname{Id}.
\end{equation*}
By \eqref{eq:minmaxF}, this also gives $c\leq F\leq C$.

Using the representation of $\phi$ in
Section~\ref{subsec:monotonicity}, we have
\begin{equation}\label{eq:phiBounds}
 \left(\min_{M_t}F\right)^2
 \leq\phi(t)\leq
 \left(\max_{M_t}F\right)^2.
\end{equation}
Thus $\phi$ has uniform positive upper and lower bounds. In particular,
\eqref{eq:polarTime} shows that $t$ and $\tau$ are uniformly comparable.
Finally, the positive inradius bounds for the original and polar domains,
together with \eqref{eq:radiusdual}, give the asserted circumradius
bounds.
\end{proof}

\subsection{Higher regularity and long-time existence}
\label{sec:regularity}

We use radial graphs as in
\cite{AndrewsChenWei}*{Section~6.4} and
\cite{AndrewsWei}*{Section~5}. Lemma~\ref{prop:two} and
\eqref{eq:shortsupport} provide an interior point $p$ which is fixed
on each time interval of a uniform length. On such an interval,
write $M_t$ as a radial graph $r(\cdot,t)$ about $p$.
Let $D$ be the Levi-Civita connection of $\sigma$, and set
\begin{equation*}
 r_i=D_i r,\qquad r_{ij}=D_iD_jr,\qquad
 v=\sqrt{1+\frac{|Dr|_\sigma^2}{\sin^2r}}.
\end{equation*}
The standard graph formulas \cite{CGLS}*{(2.18)--(2.21)} give
\begin{equation}\label{eq:radial}
\begin{aligned}
 g_{ij}&=\sin^2r\,\sigma_{ij}+r_ir_j,\\
 h_{ij}&=\frac1v\left(-r_{ij}+\sin r\cos r\,\sigma_{ij}
                                  +2\cot r\,r_ir_j\right),\\
 u&=\frac{\sin r}{v},\qquad
 d\mu=\sin^nr\,v\,d\mu_\sigma.
\end{aligned}
\end{equation}
The Weingarten map is $h^i{}_j=g^{ik}h_{kj}$.
By \eqref{eq:supportBound}, $r$ is uniformly separated from $0$
and $\pi$, and $v$ is uniformly bounded. The curvature bounds
\eqref{eq:two} and \eqref{eq:radial} therefore give
$\|r(\cdot,t)\|_{C^2(\mathbb S^n)}\leq C$.

The radial velocity has normal component $r_t/v$, so
\eqref{eq:speed} becomes
\begin{equation}\label{eq:radialFlow}
 \partial_t r=v\left(\frac{\phi(t)}F-F\right)=-vG.
\end{equation}
In particular, $\|\partial_t r\|_{C^0(\mathbb S^n)}\leq C$.
In a local $\sigma$-orthonormal frame, the symmetric matrix
$g^{-1/2}hg^{-1/2}$ is affine in $D^2r$ when $r$ and $Dr$ are
fixed. Thus Lemma~\ref{lem:Gconcavity} implies that the right-hand
side of \eqref{eq:radialFlow} is convex in $D^2r$. The graph bounds
and \eqref{eq:two} give uniform parabolicity.

There is one difference from the scalar equations in
\cite{AndrewsChenWei}*{(6.14)} and \cite{AndrewsWei}*{(5.9)}.
In those equations the global term is additive and does not enter the
local Hessian derivative of the right-hand side. Here, by
\eqref{eq:compositeDerivatives}, the leading coefficient in
\eqref{eq:linearizedOperator} is
\begin{equation*}
 \dot G^{ij}=\left(1+\frac{\phi(t)}{F^2}\right)\dot F^{ij},
\end{equation*}
which depends explicitly on $\phi(t)$.
The bounds \eqref{eq:two} give uniform ellipticity but do not yet
provide a uniform time H\"older estimate for this coefficient.
We therefore establish that estimate before applying Schauder theory.

By \eqref{eq:Fscalar}, \eqref{eq:reciprocal}, and \eqref{eq:two},
\begin{equation*}
 (\partial_t-\mathcal L)F\leq C,
 \qquad
 (\partial_t-\mathcal L)F^{-1}\leq C,
\end{equation*}
where $\mathcal L$ is defined in \eqref{eq:linearizedOperator}.
In fixed radial coordinates, both inequalities have the same uniformly
parabolic operator with bounded drift.
The standard oscillation argument based on the weak Harnack inequality
\cite{KrylovSafonov}, applied to this pair of inequalities, gives a
uniform parabolic H\"older estimate for $F$ on smaller cylinders.
This step uses only bounded measurable coefficients and requires no
time modulus for $\phi$.
Using \eqref{eq:radial} and the concavity of $F$, we apply the elliptic
Evans--Krylov estimates on each time slice
(see Evans \cite{Evans} and Caffarelli and Cabr\'e
\cite{CaffarelliCabre}*{Theorem~8.1}) to obtain uniform spatial
$C^{2+\beta}$ bounds for $r$, for some $0<\beta<1$.
Interpolation with the bound for $\partial_t r$ yields time
H\"older continuity of $D^2r$ with exponent $\beta/(2+\beta)$.
The integral formula \eqref{eq:mainflow} gives the same estimate for
$\phi$, since its integrands depend smoothly on $r$, $Dr$, and $D^2r$,
and its denominator is uniformly positive by \eqref{eq:two} and
\eqref{eq:radial}.

With these H\"older estimates, spatial differentiation of the radial
equation and parabolic Schauder estimates \cite{Lieberman} give all
higher spatial estimates, as in \cite{AndrewsWei}*{Section~5}.
No time derivative of $\phi$ is needed in this step, since $\phi$ is
spatially constant. Successive differentiation of the graph equation
and the normalization then controls all time derivatives. In particular,
for every integer $j\geq0$,
\begin{equation}\label{eq:allest}
 \sup_{0\leq t<T_{\max}}
 \left(\sup_{M_t}|\nabla^j\mathcal W|
       +|\phi^{(j)}(t)|\right)\leq C_j,
\end{equation}
where $C_j$ depends only on $M_0$, $n$, $k$, and $j$.
All interior estimates are taken on shortened overlapping graph
intervals, with the initial interval controlled by short-time existence. Thus the constants are independent of $T_{\max}$.
If $T_{\max}<\infty$, the bounds for $\partial_t r$ in every spatial
$C^j$ norm give a smooth limit on a terminal fixed-center graph interval. By \eqref{eq:two}, the limit is strictly convex. The local theory in Section~\ref{subsec:main} extends the solution beyond $T_{\max}$, a contradiction. Hence $T_{\max}=\infty$.

\section{Convergence and the Alexandrov--Fenchel inequalities}\label{sec:convergence}

In this section we prove the smooth exponential convergence of
\eqref{eq:mainflow} and complete the proofs of
Theorems~\ref{thm:flow} and~\ref{thm:AF}. The improved pinching
estimate and the higher order estimates give exponential decay of the
trace-free second fundamental form and of the normal speed. This
yields convergence in fixed radial coordinates to a geodesic sphere.
The Alexandrov--Fenchel inequalities then follow from the monotonicity
formula, together with a short-time mean curvature flow approximation
for weakly convex domains.

\subsection{Exponential decay}

\begin{lemma}\label{lem:decay}
Along the flow \eqref{eq:mainflow}, for every integer $j\geq0$,
there exist constants $C_j,\lambda_j>0$, depending only on
$M_0$, $n$, $k$, and $j$, such that
\begin{equation}\label{eq:decay}
 \sup_{M_t}|\nabla^j\mathring{\mathcal W}|
 +\sup_{M_t}|\nabla^jG|
 \leq C_je^{-\lambda_jt}
 \qquad\text{for all }t\geq0.
\end{equation}
Moreover, there exist constants $C,\lambda>0$, depending only on
$M_0$, $n$, and $k$, such that
\begin{equation}\label{eq:FoscDecay}
 \operatorname{osc}_{M_t}F\leq Ce^{-\lambda t}
 \qquad\text{for all }t\geq0.
\end{equation}
\end{lemma}

\begin{proof}
Choose
\begin{equation*}
 \varepsilon_0
 =\frac12\min_{M_0}\frac{\kappa_1}{H}>0
\end{equation*}
and take $\gamma=n\varepsilon_0$ in
Lemma~\ref{lem:improvedPinch}. It follows from \eqref{eq:improve} and  $\sum_{i=1}^n\left(\kappa_i-\frac{H}{n}\right)=0$ that
\begin{equation*}
-\left(\frac1n-\varepsilon_0\right)He^{-\gamma t} \leq \kappa_i-\frac{H}{n} \leq (n-1)\left(\frac1n-\varepsilon_0\right)He^{-\gamma t}.
\end{equation*}
This together with \eqref{eq:two} gives
\begin{equation}\label{eq:tracefree}
 \sup_{M_t}|\mathring{\mathcal W}|
 \leq Ce^{-\gamma t}.
\end{equation}
Combining \eqref{eq:tracefree} with the higher order estimates
\eqref{eq:allest} and interpolation, we obtain
\begin{equation}\label{eq:tracefreeHigher}
 \sup_{M_t}|\nabla^j\mathring{\mathcal W}|
 \leq C_je^{-\gamma_jt}
 \qquad(j\geq 0)
\end{equation}
for some $C_j,\gamma_j>0$.

The contracted Codazzi equation gives 
\begin{equation*}\nabla_j(\mathring{\mathcal W})^j{}_i=\frac{n-1}{n}\nabla_i H.
\end{equation*}
Hence \eqref{eq:tracefreeHigher} implies
\begin{equation*}
 \sup_{M_t}|\nabla H|
 \leq Ce^{-\gamma_1t}.
\end{equation*}
By the Gauss equation and \eqref{eq:two},
\begin{equation*}
 \operatorname{Ric}(e_i,e_i)
 =\sum_{j\neq i}(1+\kappa_i\kappa_j)
 \geq (n-1)(1+C^{-2}).
\end{equation*}
The Bonnet--Myers theorem therefore gives a uniform upper bound for
$\operatorname{diam}(M_t)$. Consequently,
\begin{equation}\label{eq:HoscDecay}
 \operatorname{osc}_{M_t}H
 \leq
 \operatorname{diam}(M_t)\sup_{M_t}|\nabla H|
 \leq Ce^{-\gamma_1t}.
\end{equation}

By the homogeneity of $F$ and \eqref{eq:minmaxF}, $F\left(\frac{H}{n}\operatorname{Id}\right)=\frac{H}{n}$. Moreover, \eqref{eq:two} implies that both $\mathcal W$ and $\frac{H}{n}\operatorname{Id}$ remain in a fixed compact subset of $\Gamma_+$. Hence the mean value theorem gives
\begin{equation}\label{eq:Fumbilic}
 \left|F-\frac{H}{n}\right|
 =
 \left|
 F(\mathcal W)
 -
 F\left(\frac{H}{n}\operatorname{Id}\right)
 \right|
 \leq
 C\left|
 \mathcal W-\frac{H}{n}\operatorname{Id}
 \right|
 =
 C|\mathring{\mathcal W}|.
\end{equation}
Combining this with \eqref{eq:tracefree} and
\eqref{eq:HoscDecay}, we obtain
\begin{equation}\label{eq:FoscEstimate}
 \operatorname{osc}_{M_t}F
 \leq Ce^{-\lambda t}
\end{equation}
for some $\lambda>0$. This proves \eqref{eq:FoscDecay}.

It remains to estimate the normal speed $G$. Set
\begin{equation*}
 F_-(t)=\min_{M_t}F,
 \qquad
 F_+(t)=\max_{M_t}F.
\end{equation*}
By \eqref{eq:phiBounds},
\begin{equation*}
 F_-^2(t)\leq\phi(t)\leq F_+^2(t).
\end{equation*}
Using \eqref{eq:two} and \eqref{eq:FoscEstimate}, we find
\begin{equation}\label{eq:speeddecay}
 |G|
 =\frac{|\phi-F^2|}{F}
 \leq
 \frac{F_+^2-F_-^2}{F_-}
 \leq Ce^{-\lambda t}.
\end{equation}
Since $\phi$ is spatially constant, the identity $G=F-\phi F^{-1}$ together with \eqref{eq:two} and \eqref{eq:allest} gives uniform
bounds for all spatial derivatives of $G$. Interpolation with
\eqref{eq:speeddecay} yields
\begin{equation*}
 \sup_{M_t}|\nabla^jG|
 \leq C_je^{-\widetilde\gamma_jt}
 \qquad(j\geq0).
\end{equation*}
Combining this with \eqref{eq:tracefreeHigher}, and decreasing the
decay rates if necessary, proves \eqref{eq:decay}.
\end{proof}

\subsection{Convergence in fixed radial coordinates}

\begin{lemma}\label{lem:limit}
Along the flow \eqref{eq:mainflow}, the hypersurfaces $M_t$ converge
smoothly and exponentially to a geodesic sphere. More precisely, there
exist a point $p\in\mathbb S^{n+1}$ and a time $t_0>0$ such that, for
$t\geq t_0$, $M_t$ is a radial graph $r(\cdot,t)$ about $p$ and
\begin{equation}\label{eq:graphConvergence}
 \|r(\cdot,t)-r_\infty\|_{C^j(\mathbb S^n)}
 \leq C_je^{-\lambda_jt}
 \qquad(j\geq0),
\end{equation}
where $r_\infty$ parametrizes the limiting geodesic sphere and
$C_j,\lambda_j>0$ depend only on $M_0$, $n$, $k$, and $j$.
\end{lemma}

\begin{proof}
Let $r_0>0$ be the uniform lower bound for the inradius given by
Lemma~\ref{prop:two}. By Lemma~\ref{lem:decay},
\begin{equation*}
 \int_t^\infty\sup_{M_\zeta}|G|\,d\zeta
 \leq Ce^{-\lambda t}\longrightarrow0
 \qquad\text{as }t\to\infty.
\end{equation*}
Choose $t_0>0$ so large that
\begin{equation*}
 \int_{t_0}^\infty\sup_{M_\zeta}|G|\,d\zeta
 <\frac{r_0}{2},
\end{equation*}
and take $B_{r_0}(p)\subset\Omega_{t_0}$. Since
$\partial_tX=-G\nu$, for every $t\geq t_0$,
\begin{align*}
 \operatorname{dist}_{\mathbb S^{n+1}}(p,M_t)
 &\geq
 \operatorname{dist}_{\mathbb S^{n+1}}(p,M_{t_0})
 -
 \sup_{x\in M}
 d_{\mathbb S^{n+1}}
 \bigl(X(x,t),X(x,t_0)\bigr)\\
 &\geq
 r_0-\int_{t_0}^t\sup_{M_\zeta}|G|\,d\zeta
 >\frac{r_0}{2}.
\end{align*}
Since $p\in\Omega_{t_0}$, continuity of the flow implies that
$p\in\Omega_t$ for all $t\geq t_0$. Hence by \eqref{eq:supportBound},
\begin{equation}\label{eq:lateSupport}
 B_{r_0/2}(p)\subset\Omega_t,
 \qquad
 u\geq\sin\frac{r_0}{2}
 \qquad(t\geq t_0).
\end{equation}
Thus $M_t$ can be represented as a radial graph $r(\cdot,t)$ about
the same point $p$ for all $t\geq t_0$. By \eqref{eq:lateSupport} and the estimates established in
Section~\ref{sec:regularity}, these radial graphs satisfy uniform
$C^j$ bounds for every $j\geq0$.

By \eqref{eq:radialFlow}, Lemma~\ref{lem:decay}, and the uniform
estimates above,
\begin{equation*}
 \|\partial_t r(\cdot,t)\|_{C^j(\mathbb S^n)}
 \leq C_je^{-\lambda_jt}
 \qquad(j\geq0).
\end{equation*}
Hence, for $t_2>t_1\geq t_0$,
\begin{equation*}
\begin{aligned}
 \|r(\cdot,t_2)-r(\cdot,t_1)\|_{C^j(\mathbb S^n)}
 &\leq
 \int_{t_1}^{t_2}
 \|\partial_t r(\cdot,t)\|_{C^j(\mathbb S^n)}\,dt\\
 &\leq C_je^{-\lambda_jt_1}.
\end{aligned}
\end{equation*}
Thus $r(\cdot,t)$ converges smoothly to a function $r_\infty$.
Letting $t_2\to\infty$ in the preceding estimate gives
\eqref{eq:graphConvergence}. By the two-sided curvature estimate
\eqref{eq:two}, the limiting hypersurface $M_\infty$ is strictly
convex.

By \eqref{eq:tracefree}, $M_\infty$ is totally umbilical. More
precisely,
\begin{equation*}
 \mathcal W_\infty=\kappa\operatorname{Id}.
\end{equation*}
Since $M_\infty$ is strictly convex, $\kappa>0$. The Codazzi equation
and $n\geq2$ imply that $\kappa$ is constant. Therefore $M_\infty$
is a geodesic sphere of radius $R=\operatorname{arccot}\kappa\in(0,\pi/2)$ with some center $p_\infty\in\mathbb S^{n+1}$.
\end{proof}

\subsection{Proof of the main theorems}\label{subsec:strictAF}

\begin{proof}[Proof of Theorem~\ref{thm:flow}]
By Lemma~\ref{prop:convexity}, strict convexity is preserved along the flow \eqref{eq:speed}. Lemma~\ref{prop:two} and the higher order estimates in Section~\ref{sec:regularity} imply long-time existence. By Lemma~\ref{lem:limit}, $M_t$ converges smoothly and exponentially to a geodesic sphere $\partial B_{R_\infty}$. Since $\mathcal A_{k-1}$ is preserved,
\begin{equation*}
 f_{k-1}(R_\infty)
 =
 \lim_{t\to\infty}\mathcal A_{k-1}(\Omega_t)
 =
 \mathcal A_{k-1}(\Omega_0).
\end{equation*}
The strict monotonicity of $f_{k-1}$ gives
\begin{equation*}
 R_\infty
 =
 f_{k-1}^{-1}\bigl(\mathcal A_{k-1}(\Omega_0)\bigr).
\end{equation*}
This completes the proof.
\end{proof}

\begin{proof}[Proof of Theorem~\ref{thm:AF}]
We divide the proof into two steps.

\textbf{Step 1. The strictly convex case.}
Assume first that $\Omega$ is strictly convex. Fix
$j\in\{0,\ldots,n-1\}$ and evolve $M=\partial\Omega$ by
\eqref{eq:mainflow} with $k=j$. By Theorem~\ref{thm:flow}, the
solution exists for all time and converges smoothly to a geodesic
sphere of radius $R_\infty\in(0,\pi/2)$. Since
$\mathcal A_{j-1}$ is preserved along the flow,
\begin{equation*}
 f_{j-1}(R_\infty)=\mathcal A_{j-1}(\Omega).
\end{equation*}
On the other hand, Proposition~\ref{prop:moment} shows that
$\mathcal A_j$ is nonincreasing. Hence
\begin{equation}\label{eq:adjacent}
 \mathcal A_j(\Omega)
 \geq
 f_j(R_\infty)
 =
 f_j\circ f_{j-1}^{-1}
 \bigl(\mathcal A_{j-1}(\Omega)\bigr).
\end{equation}

For $-1\leq i\leq n-1$, set
\begin{equation*}
 r_i=f_i^{-1}\bigl(\mathcal A_i(\Omega)\bigr).
\end{equation*}
Since the functions $f_i$ are strictly increasing,
\eqref{eq:adjacent} gives
\begin{equation}\label{eq:radiiMonotonicity}
 r_{-1}\leq r_0\leq\cdots\leq r_{n-1}.
\end{equation}
Therefore, for every $-1\leq\ell<k\leq n-1$,
\begin{equation*}
 \mathcal A_k(\Omega)
 =f_k(r_k)
 \geq f_k(r_\ell)
 =
 f_k\circ f_\ell^{-1}
 \bigl(\mathcal A_\ell(\Omega)\bigr).
\end{equation*}
This proves \eqref{eq:AF} for strictly convex domains.

Suppose now that equality holds in \eqref{eq:AF} for some
$-1\leq\ell<k\leq n-1$. Since $f_k$ is strictly increasing,
$r_\ell=r_k$. By \eqref{eq:radiiMonotonicity},
\begin{equation*}
 r_\ell=r_{\ell+1}=\cdots=r_k.
\end{equation*}
Hence equality holds in each adjacent inequality between $\ell$ and
$k$. Fix such an adjacent pair and write
\begin{equation*}
 \mathcal A_j(\Omega)
 =
 f_j\circ f_{j-1}^{-1}
 \bigl(\mathcal A_{j-1}(\Omega)\bigr).
\end{equation*}
Along the flow \eqref{eq:mainflow} with $k=j$,
$\mathcal A_{j-1}$ is preserved and $\mathcal A_j$ is nonincreasing.
The limiting sphere has radius $r_{j-1}=r_j$, and therefore
\begin{equation*}
 \lim_{t\to\infty}\mathcal A_j(\Omega_t)
 =
 f_j(r_{j-1})
 =
 f_j(r_j)
 =
 \mathcal A_j(\Omega).
\end{equation*}
Thus $\mathcal A_j(\Omega_t)$ is constant. Proposition~\ref{prop:moment}
then implies that $M=M_0$ is a geodesic sphere. Hence $\Omega$ is a
geodesic ball. Conversely, geodesic balls attain equality in
\eqref{eq:AF}.

\textbf{Step 2. The weakly convex case.}
Let now $\Omega$ be weakly convex and write $M=\partial\Omega$.
We evolve $M$ by the mean curvature flow
\begin{equation*}
 \partial_tX=-H\nu,
 \qquad
 X(M,0)=M.
\end{equation*}
By the short-time existence theory \cite{HuiskenPolden}, there is a
smooth solution on $[0,T)$ for some $T>0$. Weak convexity is
preserved, and the strong maximum principle implies that $M_t$ is
strictly convex for every $t>0$, see
\cite{BryanIvaki}*{Section~3.1}. The equator case is excluded since
$M$ is contained in an open hemisphere. After decreasing $T$ if
necessary, $M_t=\partial\Omega_t$ remains in the same open hemisphere
and converges smoothly to $M$ as $t\downarrow0$. Thus
\begin{equation*}
 \mathcal A_j(\Omega_t)\longrightarrow\mathcal A_j(\Omega)
 \qquad
 (-1\leq j\leq n-1).
\end{equation*}
Applying the inequality proved in \textbf{Step~1} to $\Omega_t$ and letting
$t\downarrow0$, we obtain
\begin{equation*}
 \mathcal A_k(\Omega)
 \geq
 f_k\circ f_\ell^{-1}
 \bigl(\mathcal A_\ell(\Omega)\bigr)
\end{equation*}
for every $-1\leq\ell<k\leq n-1$. This proves \eqref{eq:AF} for
weakly convex domains.

It remains to consider the equality case. The adjacent inequalities
also pass to the limit above. As in \textbf{Step~1}, equality in
\eqref{eq:AF} for a pair $\ell<k$ forces equality in each adjacent
comparison between $\ell$ and $k$. It is therefore enough to assume
that, for some $0\leq k\leq n-1$,
\begin{equation}\label{eq:weakAdjacentEquality}
 \mathcal A_k(\Omega)
 =
 f_k\circ f_{k-1}^{-1}
 \bigl(\mathcal A_{k-1}(\Omega)\bigr).
\end{equation}

Set
\begin{equation*}
 R=f_{k-1}^{-1}\bigl(\mathcal A_{k-1}(\Omega)\bigr),
 \qquad
 c=\cot R>0,
\end{equation*}
and let
\begin{equation*}
 M_+=\{x\in M:\mathcal W(x)>0\}.
\end{equation*}
The set $M_+$ is nonempty and open. Indeed, choose
$p_0\in\mathbb S^{n+1}$ and $R_0<\pi/2$ such that
\begin{equation*}
 \Omega\subset B_{R_0}(p_0),
 \qquad
 M\cap\partial B_{R_0}(p_0)\neq\varnothing.
\end{equation*}
At a contact point the second fundamental form comparison gives
\begin{equation*}
 \mathcal W\geq\cot R_0\operatorname{Id}>0.
\end{equation*}

We claim that $M_+$ is also closed. Let
$\eta\in C_c^\infty(M_+)$ and consider the normal variation
\begin{equation*}
 X_s(x)=\exp_{X(x)}\bigl(s\eta(x)\nu(x)\bigr).
\end{equation*}
Since $\operatorname{supp}\eta$ is compactly contained in $M_+$,
for sufficiently small $|s|$ the hypersurface $M_s$ remains embedded
in the same open hemisphere and weakly convex. It therefore bounds a
weakly convex domain $\Omega_s$.

Set
\begin{equation*}
 Q(s)
 =
 \mathcal A_k(\Omega_s)
 -
 f_k\circ f_{k-1}^{-1}
 \bigl(\mathcal A_{k-1}(\Omega_s)\bigr).
\end{equation*}
By \eqref{eq:AF} and \eqref{eq:weakAdjacentEquality},
\begin{equation*}
 Q(0)=0,
 \qquad
 Q(s)\geq0
\end{equation*}
for sufficiently small $|s|$. Hence $\left.\frac{d}{ds}\right|_{s=0}Q(s)=0$. Applying \eqref{eq:Avariation} together with the chain rule yields
\begin{align*}
 0=\left.\frac{d}{ds}\right|_{s=0}Q(s) 
 =&
 d_{k+1}\int_M\eta E_{k+1}\,d\mu
 -\frac{f_k'(R)}{f_{k-1}'(R)}
 d_k\int_M\eta E_k\,d\mu \\
 =& d_{k+1}\int_M \eta\bigl(E_{k+1}-cE_k\bigr)\,d\mu,
\end{align*}
where the last equality follows from \eqref{eq:ball}, which gives
\begin{equation*}
 \frac{f_k'(R)}{f_{k-1}'(R)}=\frac{d_{k+1}}{d_k}\cot R=\frac{d_{k+1}}{d_k}c.
\end{equation*}
Since $\eta\in C_c^\infty(M_+)$ is arbitrary and $E_k>0$ on $M_+$, we conclude that
\begin{equation}\label{eq:weakEuler}
 F=\frac{E_{k+1}}{E_k}=c
 \qquad\text{on }M_+.
\end{equation}

By the Newton--MacLaurin inequalities
\eqref{eq:NewtonMaclaurin} and \eqref{eq:weakEuler},
\begin{equation*}
 E_k\geq F^k=c^k
 \qquad\text{on }M_+,
\end{equation*}
where $E_0=1$ when $k=0$. Hence the same lower bound holds on
$\overline{M_+}$. In particular, $E_k>0$ in a neighborhood of
$\overline{M_+}$, so $F=E_{k+1}/E_k$ is smooth there.
The identity \eqref{eq:weakEuler} and its first two spatial
derivatives therefore extend to $\overline{M_+}$.

Suppose that $x_0\in\overline{M_+}\setminus M_+$. Choose an
orthonormal principal frame at $x_0$ such that
\begin{equation*}
 0=\kappa_1=\cdots=\kappa_\mu
 <\kappa_{\mu+1}\leq\cdots\leq\kappa_n.
\end{equation*}
Since $F(x_0)=c>0$, we have $\mu<n$. Applying
Lemma~\ref{lem:BCD} with $\varphi\equiv0$, together with
\eqref{eq:BCDfirst} and the Codazzi equation, gives
\begin{equation}\label{eq:weakNullGradient}
 \nabla_i h_{1j}
 =
 \nabla_1h_{ij}
 =
 0
 \qquad(j\leq\mu).
\end{equation}
Moreover, \eqref{eq:BCDsecond} gives
\begin{equation}\label{eq:weakNullHessian}
 \nabla_i\nabla_i h_{11}
 \geq
 2\sum_{j>\mu}
 \frac{(\nabla_i h_{1j})^2}{\kappa_j},
\end{equation}
where there is no summation over $i$.

Since the dual function $F_*$ in \eqref{eq:dualF} is concave, the
inverse-concavity inequality \cite{Andrews}*{Lemma~3.4}, applied to
$\mathcal W+\delta\operatorname{Id}$ and followed by
$\delta\downarrow0$, gives
\begin{equation}\label{eq:weakInverseConcavity}
 \ddot F^{ij,pq}\nabla_1h_{ij}\nabla_1h_{pq}
 +
 2\sum_i\sum_{j>\mu}
 \frac{\dot F^i}{\kappa_j}
 (\nabla_1h_{ij})^2
 \geq0.
\end{equation}
Here \eqref{eq:weakNullGradient} removes the terms corresponding to
the zero eigenvalues. Moreover, $\dot F^i\geq0$ by continuity from
$\Gamma_+$.

Multiplying \eqref{eq:weakNullHessian} by $\dot F^i$, summing over
$i$, and using the Codazzi equation and
\eqref{eq:weakInverseConcavity}, we obtain
\begin{equation}\label{eq:weakQ}
 \dot F^{ij}\nabla_i\nabla_jh_{11}
 +
 \ddot F^{ij,pq}\nabla_1h_{ij}\nabla_1h_{pq}
 \geq0.
\end{equation}
On the other hand, \eqref{eq:weakEuler} gives
\begin{equation*}
 \nabla_1\nabla_1F(x_0)=0.
\end{equation*}
Since $h_{11}=(h^2)_{11}=0$ at $x_0$, the Simons identity
\eqref{eq:Simons} yields
\begin{equation*}
\begin{aligned}
 0
 &=\nabla_1\nabla_1F\\
 &=\dot F^{ij}\nabla_i\nabla_jh_{11}
   +\ddot F^{ij,pq}\nabla_1h_{ij}\nabla_1h_{pq}
   +F\\
 &\geq c>0,
\end{aligned}
\end{equation*}
a contradiction. Thus $M_+$ is closed. Since $M$ is connected,
$M_+=M$. Hence $M$ is strictly convex. The equality case proved in
\textbf{Step~1} now implies that $\Omega$ is a geodesic ball.

This completes the proof.
\end{proof}

\appendix
\section{A curvature blow-up example for the locally constrained flow}
\label{sec:counterexample}

We construct a finite-time curvature blow-up example for the locally
constrained flow of Chen, Guan, Li, and Scheuer \cite{CGLS}.
The argument uses spherical polarity. A weakly convex polar
hypersurface has a zero principal curvature with negative initial
time derivative, and small inward parallel perturbations give
strictly convex initial data whose polar evolutions lose strict
convexity in finite time.

\begin{proposition}\label{prop:CGLSblowup}
Let $n\geq2$ and $1\leq k\leq n-1$. For any
$p\in\mathbb S^{n+1}$, there exists a smooth closed rotationally
symmetric strictly convex hypersurface $M_0\subset B_{\pi/2}(p)$
enclosing $p$ such that the flow
\begin{equation}\label{eq:appDirect}
 \partial_tX=c_{n,k}(\cos r-uF)\nu,
 \qquad
 F=\frac{E_{k+1}}{E_k},
 \qquad
 c_{n,k}=\frac{n-k}{k+1},
\end{equation}
has a finite maximal smooth existence time $T$ and satisfies
\begin{equation}
 \lim_{t\uparrow T}\max_{M_t}\kappa_{\max}=+\infty.
 \notag
\end{equation}
Here $r$ and $u$ are measured from $p$.
\end{proposition}

\begin{proof}
Set $p^*=-p$ and $\ell=n-k$. With the polar convention
$Y=\nu$ and $N=X$, we have $\mathcal W^*=\mathcal W^{-1}$.
Writing
\begin{equation}
 q=\frac{E_{\ell-1}(\mathcal W^*)}{E_\ell(\mathcal W^*)},
 \notag
\end{equation}
the polarity identities in Section~\ref{subsec:polarity} give
\begin{equation}
 F=q,\qquad \cos r=u^*,\qquad u=\cos r^*,
 \notag
\end{equation}
where $r^*$ and $u^*$ are measured from $p^*$. Thus, up to a
tangential reparametrization, the polar flow is
\begin{equation}\label{eq:appPolar}
 \partial_tY=\eta N,
 \qquad
 \eta=c_{n,k}(\cos r^*\,q-u^*).
\end{equation}
In radial coordinates about $p^*$, this becomes
\begin{equation}\label{eq:appScalar}
 \partial_t r^*
 =c_{n,k}\left(v^*\cos r^*\,q-\sin r^*\right),
 \qquad
 v^*=\sqrt{1+\frac{|Dr^*|^2}{\sin^2r^*}}.
\end{equation}

We first construct the weakly convex initial polar hypersurface.
Choose an orthonormal basis $\{E_1,\ldots,E_{n+1}\}$ of
$T_{p^*}\mathbb S^{n+1}=(p^*)^\perp\subset\mathbb R^{n+2}$.
Using this basis, identify the tangent space with
$\mathbb R^{n+1}=\mathbb R\times\mathbb R^n$ and the vector
$\xi=\sum_{i=1}^{n+1}x_iE_i$ with its coordinates, writing
\begin{equation}
 \xi=(x_1,x'),
 \qquad
 x'=(x_2,\ldots,x_{n+1})\in\mathbb R^n.
 \notag
\end{equation}
Let $\chi:\mathbb R\to[0,\infty)$ be a smooth even convex
function such that
\begin{equation}
 \chi=0\ \text{on }[-3,3],\qquad
 \chi''>0\ \text{for }|x_1|>3,\qquad
 \chi(x_1)\to+\infty\ \text{as }|x_1|\to\infty.
 \notag
\end{equation}
Define
\begin{equation}
 \Psi(\xi)=\chi(x_1)+|x'|^2-1,
 \qquad
 D=\{\xi\in\mathbb R^{n+1}:\Psi(\xi)\leq0\}.
 \notag
\end{equation}
Then $D$ is a compact convex body containing the origin in its
interior. Its boundary is smooth because
$D\Psi=(\chi'(x_1),2x')$ does not vanish on $\{\Psi=0\}$.
Geometrically, $\partial D$ consists of a cylindrical portion
$[-3,3]\times\mathbb S^{n-1}$ with smoothly attached strictly
convex ends.

Map $D$ into $B_{\pi/2}(p^*)$ by
\begin{equation}
 \Pi(\xi)=\frac{p^*+\xi}{\sqrt{1+|\xi|^2}},
 \qquad
 \Omega_0^*=\Pi(D),
 \qquad
 M_0^*=\Pi(\partial D).
 \notag
\end{equation}
Since $\Pi$ maps line segments to geodesic segments,
$\Omega_0^*$ is geodesically convex. Set
\begin{equation}
 L=\sqrt{1+|\xi|^2},
 \qquad
 Q=\sqrt{|D\Psi|^2+(\xi\cdot D\Psi)^2}.
 \notag
\end{equation}
For $V=(V_1,V')\in T_\xi\partial D$, a direct calculation gives
\begin{equation}
 h_0^*(d\Pi V,d\Pi V)
 =\frac{\chi''(x_1)V_1^2+2|V'|^2}{LQ}.
 \notag
\end{equation}
Thus $M_0^*$ is weakly convex and has exactly one zero principal
curvature on the image of
\begin{equation}\label{eq:appRuledRegion}
 |x_1|\leq3,\qquad |x'|=1,
\end{equation}
while all principal curvatures are positive outside this region.
Moreover,
\begin{equation}
 \cos r^*=\frac1L>0,
 \qquad
 u^*=\frac{x_1\chi'(x_1)+2|x'|^2}{Q}>0.
 \notag
\end{equation}

On \eqref{eq:appRuledRegion}, write
$\xi=x_1E_1+\omega$, where $x_1=\sqrt2\cot s$ and
$\omega\in\mathbb S^{n-1}\subset
\operatorname{span}\{E_2,\ldots,E_{n+1}\}$, with
$s\in(0,\pi)$ and $|\sqrt2\cot s|\leq3$.
The position vector and outward unit normal are
\begin{equation}
 Y(s,\omega)
 =\cos s\,E_1+\frac{\sin s}{\sqrt2}(p^*+\omega),
 \qquad
 N(s,\omega)=\frac{\omega-p^*}{\sqrt2}.
 \notag
\end{equation}
Since
\begin{equation}
 |Y_s|=1,\qquad Y_{ss}=-Y,\qquad N_s=0,\qquad
 d_\omega N=\frac1{\sin s}\,d_\omega Y,
 \notag
\end{equation}
the principal curvatures and support quantities on this region are
\begin{equation}
 \begin{gathered}
 \kappa_1^*=0,\qquad
 \kappa_2^*=\cdots=\kappa_n^*=a(s)=\frac1{\sin s},\\
 \cos r^*=\frac{\sin s}{\sqrt2},\qquad u^*=\frac1{\sqrt2}.
 \end{gathered}
 \notag
\end{equation}
At a curvature vector $(0,a,\ldots,a)$, we also have
\begin{equation}\label{eq:appQvalues}
 q=\frac{k+1}{ka},\qquad
 -\dot q^1=\frac{n}{k^2a^2},\qquad
 -\dot q^j=\frac{k+1}{(n-1)ka^2},\quad 2\leq j\leq n.
\end{equation}
Since $\ell\leq n-1$, $E_\ell>0$ on $M_0^*$.
The positivity of $\cos r^*$, \eqref{eq:appQvalues}, and
ellipticity on the strictly convex part therefore give, by
compactness, an open neighborhood of the initial two-jets on
which \eqref{eq:appScalar} is uniformly parabolic. In particular,
the equation remains smooth and parabolic when the smallest
principal curvature is slightly negative.

The preceding identities give the normal speed on the ruled region:
\begin{equation}
 \eta(s)=\frac{c_{n,k}}{\sqrt2}
 \left(\frac{k+1}{k}\sin^2s-1\right).
 \notag
\end{equation}
The curves $s\mapsto Y(s,\omega)$ are unit-speed geodesics.
Since $\kappa_1^*=0$, the curvature variation formula for
\eqref{eq:appPolar} yields
\begin{equation}
 \partial_t\kappa_1^*(s,0)
 =-\eta_{ss}-\eta
 =\frac{c_{n,k}}{\sqrt2}
 \left(\frac{3(k+1)}k\sin^2s-\frac{k+2}k\right).
 \notag
\end{equation}
Choose a point $o$ in the ruled region with $x_1=2$.
Then $\sin^2s=1/3$, and hence
\begin{equation}\label{eq:appNegative}
 \partial_t\kappa_1^*(o,0)
 =-\frac{n-k}{\sqrt2\,k(k+1)}
 =-\gamma_{n,k}<0.
\end{equation}

We now perturb $M_0^*$ inward. Let $Y_0$ and $N_0$ denote its
position vector and outward unit normal, and set
\begin{equation}
 Y_\varepsilon=\cos\varepsilon\,Y_0-\sin\varepsilon\,N_0,
 \qquad
 N_\varepsilon=\sin\varepsilon\,Y_0+\cos\varepsilon\,N_0.
 \notag
\end{equation}
For sufficiently small $\varepsilon>0$, the parallel hypersurface
$M_\varepsilon^*$ is smoothly embedded, and its principal
curvatures satisfy
\begin{equation}\label{eq:appParallel}
 \kappa_{i,\varepsilon}^*
 =\frac{\kappa_i^*+\tan\varepsilon}
        {1-\kappa_i^*\tan\varepsilon}.
\end{equation}
Taking $\varepsilon$ small enough that all denominators are
positive, we obtain strict convexity and
\begin{equation}
 \min_{M_\varepsilon^*}\kappa_{\min}^*=\tan\varepsilon.
 \notag
\end{equation}
These hypersurfaces converge smoothly to $M_0^*$ as
$\varepsilon\downarrow0$.

By the uniform parabolicity established above and smooth dependence
on the initial data, the solutions of \eqref{eq:appScalar} starting
from $M_\varepsilon^*$ exist on a common interval $[0,\tau_*]$
for all sufficiently small $\varepsilon\geq0$ and converge
smoothly to the solution starting from $M_0^*$ (see
\cites{HuiskenPolden,Lunardi}). Shortening $\tau_*$ if necessary,
we may assume that $\cos r^*>0$ and $u^*>0$ throughout.

Use smoothly chosen parametrizations on a common reference manifold,
with initial embeddings $Y_\varepsilon$, and let
$\lambda_\varepsilon(o,t)$ denote the principal curvature branch
at the point labeled by $o$ that vanishes at
$(\varepsilon,t)=(0,0)$. This eigenvalue is simple there and
remains simple for small $\varepsilon$ and $t$.
By \eqref{eq:appParallel},
$\lambda_\varepsilon(o,0)=\tan\varepsilon$.
For the unperturbed initial hypersurface, this curvature vanishes
identically on the ruled region, so its spatial gradient at $o$
is zero. Thus tangential reparametrization does not alter the
initial time derivative in \eqref{eq:appNegative}, and
$\partial_t\lambda_0(o,0)=-\gamma_{n,k}$.
By continuity, there exists $\tau\in(0,\tau_*)$ such that
\begin{equation}
 \partial_t\lambda_0(o,t)\leq-\frac34\gamma_{n,k},
 \qquad 0\leq t\leq\tau.
 \notag
\end{equation}
Smooth convergence of the solutions and simplicity of the branch
give $\lambda_\varepsilon(o,\cdot)\to\lambda_0(o,\cdot)$ in
$C^1([0,\tau])$. Consequently, there exists $\varepsilon_0>0$
such that
\begin{equation}
 \sup_{0\leq t\leq\tau}
 \left|\partial_t\lambda_\varepsilon(o,t)
       -\partial_t\lambda_0(o,t)\right|
 \leq\frac14\gamma_{n,k},
 \qquad 0\leq\varepsilon\leq\varepsilon_0.
 \notag
\end{equation}
It follows that, for $0\leq t\leq\tau$ and
$0\leq\varepsilon\leq\varepsilon_0$,
\begin{equation}\label{eq:appCurvatureDecay}
 \begin{aligned}
 \partial_t\lambda_\varepsilon(o,t)
 &\leq-\frac12\gamma_{n,k},\\
 \lambda_\varepsilon(o,t)
 &\leq\tan\varepsilon-\frac{\gamma_{n,k}}2t,
 \end{aligned}
\end{equation}
where the second inequality follows by integration.

Decrease $\varepsilon_0$ further so that
$3\tan\varepsilon_0/\gamma_{n,k}<\tau$, and fix
$0<\varepsilon\leq\varepsilon_0$.
Denote the evolving polar hypersurface by $M_{\varepsilon,t}^*$
and set
$m_\varepsilon(t)=\min_{M_{\varepsilon,t}^*}\kappa_{\min}^*$.
We have $m_\varepsilon(0)=\tan\varepsilon>0$, whereas
\eqref{eq:appCurvatureDecay} gives
\begin{equation}
 m_\varepsilon\left(\frac{2\tan\varepsilon}{\gamma_{n,k}}\right)
 \leq
 \lambda_\varepsilon\left(o,
       \frac{2\tan\varepsilon}{\gamma_{n,k}}\right)
 \leq0.
 \notag
\end{equation}
Thus the first time at which strict convexity is lost,
\begin{equation}
 T_\varepsilon
 :=\inf\{t\in(0,\tau]:m_\varepsilon(t)\leq0\},
 \notag
\end{equation}
satisfies
\begin{equation}\label{eq:appCrossing}
 0<T_\varepsilon
 \leq\frac{2\tan\varepsilon}{\gamma_{n,k}}
 =\frac{2\sqrt2\,k(k+1)}{n-k}\tan\varepsilon.
\end{equation}
By continuity, $m_\varepsilon(t)>0$ for $t<T_\varepsilon$
and $m_\varepsilon(T_\varepsilon)=0$.
Since \eqref{eq:appCrossing} gives $T_\varepsilon<\tau$, the
polar solution remains smooth through this time. Moreover,
\begin{equation}
 \lambda_\varepsilon\left(o,
       \frac{3\tan\varepsilon}{\gamma_{n,k}}\right)
 \leq-\frac12\tan\varepsilon<0,
 \notag
\end{equation}
so it becomes nonconvex within the common existence interval.

For $0\leq t<T_\varepsilon$, inverse polarity yields a smooth
strictly convex solution $M_t$ of \eqref{eq:appDirect}.
The positivity of $\cos r^*$ and $u^*$ ensures that the original
domains contain $p$ and are contained in $B_{\pi/2}(p)$.
Rotational symmetry is preserved by uniqueness.
Finally, $\mathcal W=(\mathcal W^*)^{-1}$ implies
\begin{equation}
 \max_{M_t}\kappa_{\max}
 =\frac1{m_\varepsilon(t)}
 \longrightarrow+\infty
 \qquad\text{as }t\uparrow T_\varepsilon.
 \notag
\end{equation}
The original solution therefore cannot be extended smoothly
through $T_\varepsilon$, and its maximal smooth existence time
is $T=T_\varepsilon<\infty$.
\end{proof}

The restriction $k\geq1$ is essential to this construction.
For $k=0$, the polar quotient is $q=E_{n-1}/E_n$, which becomes
singular when one principal curvature vanishes. The uniformly
parabolic neighborhood used above is then no longer available.

\section*{Acknowledgements}
This work was supported by the National Key Research and Development Program of China (2021YFA1001800), National Natural Science Foundation of China (No.~12531002) and the Fundamental Research Funds for the Central Universities.  The third author was also supported by the China Postdoctoral Science Foundation (Grant No.~2025M783146).

\end{document}